\documentclass[10pt,a4paper]{article}
\usepackage[utf8]{inputenc}
\usepackage{amsmath}
\usepackage{amsfonts}
\usepackage{amssymb}
\usepackage{amsthm}
\usepackage{enumitem}
\usepackage{color}
\usepackage{hyperref}
\usepackage{MnSymbol}
\usepackage{graphicx}

\usepackage[a4paper, textwidth=13cm]{geometry}

\newcommand{\N}{\mathbb{N}}
\newcommand{\R}{\mathbb{R}}
\newcommand{\Z}{\mathbb{Z}}
\newcommand{\oeo}{\Omega_{\epsilon}^1}
\newcommand{\oet}{\Omega_{\epsilon}^2}
\newcommand{\oej}{\Omega_{\epsilon}^j}

\newcommand{\geh}{\Gamma_{\epsilon,h}}
\newcommand{\oeho}{\Omega_{\epsilon,h}^1}
\newcommand{\oeht}{\Omega_{\epsilon,h}^2}
\newcommand{\oehj}{\Omega_{\epsilon,h}^j}
\newcommand{\veps}{v_{\epsilon}}
\renewcommand{\ge}{\Gamma_{\epsilon}}
\newcommand{\ueps}{u_{\epsilon}}
\newcommand{\uepsj}{u_{\epsilon}^j}
\newcommand{\uepst}{u_{\epsilon}^2}
\newcommand{\uepso}{u_{\epsilon}^1}
\newcommand{\depsj}{D^j_{\epsilon}}
\newcommand{\fepsj}{f_{\epsilon}^j}
\newcommand{\hepsj}{h_{\epsilon}^j}
\newcommand{\pepsj}{\phi_{\epsilon}^j}
\newcommand{\peps}{\phi_{\epsilon}}
\newcommand{\fxe}{\frac{x}{\epsilon}}
\newcommand{\ie}{i.\,e.,\,}
\newcommand{\per}{\mathrm{per}}
\newcommand{\ngeps}{\nabla_{\ge}}
\newcommand{\dgepsj}{D_{\ge}^j}
\newcommand{\hd}{\widehat{D}}
\newcommand{\Hjeps}{\mathbb{H}_{j,\epsilon}}
\newcommand{\Hoeps}{\mathbb{H}_{1,\epsilon}}
\newcommand{\Ljeps}{\mathbb{L}_{j,\epsilon}}
\newcommand{\Loeps}{\mathbb{L}_{1,\epsilon}}
\newcommand{\Lteps}{\mathbb{L}_{2,\epsilon}}
\newcommand{\Lj}{\mathbb{L}_j}
\newcommand{\Hj}{\mathbb{H}_j}
\newcommand{\Ho}{\mathbb{H}_1}
\renewcommand{\H}{\mathbb{H}}
\renewcommand{\oe}{\Omega_{\epsilon}}

\renewcommand{\ng}{\nabla_{\Gamma}}
\newcommand{\ngy}{\nabla_{\Gamma,y}}

\newcommand{\eg}{e.\,g.,}
\newcommand{\te}{\mathcal{T}_{\epsilon}}

\newcommand{\Hjepsh}{\mathbb{H}_{j,\epsilon,h}}
\newcommand{\pepsh}{\phi_{\epsilon,h}}
\newcommand{\bpepsh}{\overline{\phi}_{\epsilon,h}}

\newcommand{\Htepsh}{\mathbb{H}_{2,\epsilon,h}}
\newcommand{\Hteps}{\mathbb{H}_{2,\epsilon}}
\newcommand{\Ljepsh}{\mathbb{L}_{j,\epsilon,h}}
\newcommand{\Ltepsh}{\mathbb{L}_{2,\epsilon,h}}
\newcommand{\tveps}{\tilde{v}_{\epsilon}}

\newtheorem{definition}{Definition}
\newtheorem{remark}{Remark}
\newtheorem{theorem}{Theorem}
\newtheorem{proposition}{Proposition}
\newtheorem{lemma}{Lemma}
\newtheorem{corollary}{Corollary}

\title{Homogenization of a two-phase problem with nonlinear dynamic Wentzell-interface condition for connected-disconnected porous media}

\author{M. Gahn}
\date{}

\begin{document}

\maketitle

\begin{abstract}
We investigate a reaction-diffusion problem in a two-component porous medium with a nonlinear interface condition between the different components. One component is connected and the other one is disconnected. 
The ratio between the microscopic pore scale and the size of the whole domain is described by the small parameter $\epsilon$. On the interface between the components we consider a dynamic Wentzell-boundary condition, where the normal fluxes from the bulk domains are given by a reaction-diffusion equation for the traces of the bulk-solutions, including nonlinear reaction-kinetics depending on the solutions on both sides of the interface. Using two-scale techniques, we pass to the limit $\epsilon \to 0$ and derive macroscopic models, where we need homogenization results for surface diffusion. To cope with the nonlinear terms we derive strong two-scale convergence results.
\end{abstract}

\\


\let\thefootnote\relax\footnotetext{Interdisciplinary Center for Scientific Computing, University of Heidelberg, Im Neuenheimer Feld
	205, 69120 Heidelberg, Germany, markus.gahn@iwr.uni-heidelberg.de.}

\section{Introduction}

In this paper we derive homogenized models for nonlinear reaction-diffusion problems with dynamic Wentzell-boundary conditions in  multi-component porous media. The domain consists of two components $\oeo$ and $\oet$, where $\oeo$ is connected, and $\oet $ is disconnected and consists of periodically distributed inclusions. The small scaling parameter $\epsilon $ represents the ratio between the length of an inclusion an the size of the whole domain. At the interface $\ge$ between the two components we assume a dynamic Wentzell-boundary condition, \ie the normal flux at the surface is given by a reaction-diffusion equation on $\ge$. More precisely, this boundary/interface condition describes processes like reactions, adsorption, desorption, and diffusion at the interface $\ge$. Further it takes into account exchange of species between the different compartments, what can be modeled by nonlinear reaction-kinetics depending on the solutions on both sides of $\ge$. The aim is the derivation of macroscopic models with homogenized diffusion coefficients for $\epsilon \to 0$, the solution of which is an approximation of the microscopic solution. An additional focus of the paper is to provide general strong two-scale compactness results, which are based on \textit{a priori} estimates for the microscopic solution.

Reaction-diffusion processes play an important role in many applications, and our model is motivated by metabolic and regulatory processes in living cells. Here, an important example is the carbohydrate metabolism in plant cells, where biochemical species are diffusing and reacting within the (connected) cytoplasm and the (disconnected) organelles (like chloroplasts and mitochondria),  and are exchanged between different cellular compartments. At the outer mitochondrial membrane takes place the process of metabolic channeling, where intermediates in metabolic pathways are passed directly from enzyme to enzyme without equilibrating in the bulk-solution phase of the cell \cite{winkel2004metabolic}. This effect can be modeled by the dynamic Wentzell-boundary condition, see \cite[Chapter 4]{GahnDissertation} for more details about the  modeling and the derivation of these boundary conditions, which can be derived by  asymptotic analysis.

To pass to the limit $\epsilon \to 0$ in the variational equation for the microscopic problem we have to cope with several difficulties. The main challenges are the coupled bulk-surface diffusion in the perforated domains, as well as the treatment of the nonlinear terms, especially on the oscillating surface $\ge$. To overcome these problems we make use of the two-scale method in perforated domains and on oscillating surfaces, where we need two-scale compactness results for diffusion processes on surfaces. To pass to the limit in the nonlinear terms we need strong convergence results. Such results are quite standard for the connected domain, but the usual methods fail for the disconnected domain. Here we make use of the unfolding method, which gives us a characterization for the two-scale convergence via functions defined on fixed domains, and a Kolmogorov-Simon-type compactness result for Banach-valued function spaces. Additionally, due to the nonlinear structure of the problem and the weak assumptions on the data, we have to deal with low regularity for the time-derivative.

There exists a large amount of papers dealing with homogenization problems for parabolic equations in multi-component porous media. However, results for the connected-disconnected case for nonlinear problems, especially for nonlinear interface conditions, seem to be rare. In \cite{GahnNeussRaduKnabnerEffectiveModelSubstrateChanneling} and \cite{Gahn} systems of reaction-diffusion problems are considered with nonlinear interface conditions. In \cite{GahnNeussRaduKnabnerEffectiveModelSubstrateChanneling}   surface concentration is included and an additional focus lies on the modeling part of the carbohydrate metabolism and the specific structure of the nonlinear reaction kinetics.  In the present paper we extend those models to problems including an additional surface diffusion for the traces of the bulk-solutions in $\oeo$ and $\oet$. The stationary case for different scalings with a continuous normal flux condition at the interface, given by a nonlinear monotone function depending on the jump of the solutions on both sides, is treated in \cite{DonatoNguyen_HomogenizationOfDiffusion}. There, the nonlinear terms in the disconnected domain only occurs for particular scalings and it is not straightforward to generalize those results to systems.

A double porosity model, where the diffusion inside the disconnected domain is of order $\epsilon^2$, is considered in \cite{BourgeatLuckhausMikelic,Ptashnyk} for continuous transmission conditions at the interface for the solutions and the normal fluxes. In \cite{BourgeatLuckhausMikelic} a nonlinear diffusion coefficient is considered, and the convergence of the nonlinear term is obtained by using the Kirchhoff-transformation and comparing the microscopic and the macroscopic equation, where the last one was obtained by a formal asymptotic expansion. Nonlinear reaction-kinetics in the bulk domains and an additional ordinary differential equation on the interface is considered in \cite{Ptashnyk}, where the strong convergence is proved by showing that the unfolded sequence of the microscopic solution is a Cauchy-sequence. A similar model with different kind of interface conditions is considered in \cite{HornungJaegerMikelic}, where the method of two-scale convergence is used and a variational principle to identify the limits of the nonlinear terms.

To pass to the limit in the diffusive terms on the interface $\ge$ arising from the Wentzell-boundary condition, compactness results for the surface gradient on an oscillating manifolds are needed. For such kind of problems in \cite{AllaireHutridurga2012,GrafPeterHomogenizationManifolds} two-scale compactness results are derived for connected surfaces, where in \cite{GrafPeterHomogenizationManifolds} the method of unfolding is used. Compactness results for a coupled bulk-surface problem when the evolution of the trace of the bulk-solution on the surface $\ge$ is described by a diffusion equation, are treated in \cite{AmarGianni2018,Gahn2019}. In \cite{AmarGianni2018} continuity of the traces across the interface is assumed, where in \cite{Gahn2019} also jumps across the interface are allowed and also compactness results for the disconnected domain $\oet$ are derived. In \cite{AmarGianni2018}, the convergence results are applied to a linear problem with a dynamic Wentzell-interface condition.  A reaction-diffusion problem including dynamic Wentzell-boundary conditions and nonlinear reaction-rates in the bulk domain and on the surface is considered in \cite{Anguiano2020} for a connected perforated domain.

In this paper we start with the microscopic model and establish existence and uniqueness of a weak solution. The appropriate function space for a weak solution is the space of Sobolev functions of first order with $H^1$-traces on the interface $\ge$, which we denote by $\Hjeps$ for $j=1,2$. To pass to the limit $\epsilon \to 0 $ we make use of the method of two-scale convergence for domains and surfaces, see \cite{Allaire_TwoScaleKonvergenz,AllaireDamlamianHornung_TwoScaleBoundary,Neuss_TwoScaleBoundary,Nguetseng}. 
 For the treatment of the diffusive terms on the oscillating surface we use the methods developed in \cite{Gahn2019} for the spaces $\Hjeps$. Those two-scale compactness results are based on \textit{a priori} estimates for the microscopic solution depending explicitly on $\epsilon$. However, to pass to the limit in the nonlinear terms, the usual (weak) two-scale convergence is not enough and we need strong two-scale convergence, what leads to difficulties especially in the disconnected domain $\oet$. The strong convergence is obtained be applying the unfolding operator, see \cite{CioranescuGrisoDamlamian2018} for an overview of the unfolding method, to the microscopic solution and use a Kolmogorov-Simon-type compactness result for the unfolded sequence. We derive a general strong two-scale compactness result that is based only on \textit{a priori} estimates and estimates for the difference between the  solution and  discrete shifts (with respect to the microscopic cells) of the solution. Since we only take into account linear shifts, which are not well defined for general surfaces, we use a Banach-valued Kolmogorov-Simon-compactness result, see \cite{GahnNeussRaduKolmogorovCompactness}. Further, for our microscopic model we only obtain low regularity results for the time-derivative (which is a functional on $\Hjeps$), what leads to additional difficulties in the control of the time variable in the proof of the strong convergence.
 
The paper is organized as follows: In Section \ref{SectionMicroscopicModel} we introduce the geometrical setting and the microscopic model. 
 In Section \ref{SectionExistence} we show existence and uniqueness of a microscopic solution, and derive \textit{a priori} estimates depending explicitly on $\epsilon$. In Section \ref{SectionTwoScaleConvergenceUnfolding} we prove general strong two-scale compactness results for the connected and disconnected domain. In Section \ref{SectionDerivationMacroscopicModel} we state the convergence results for the microscopic solution, formulate the macroscopic model, and show that the limit of the micro-solutions solves the macro-model.
In the Appendix \ref{SectionAppendix} we repeat the definition of the two-scale convergence and the unfolding operator and summarize some well known results from the literature.

\section{The microscopic model}
\label{SectionMicroscopicModel}

In this section we introduce the microscopic problem. We start with the definition of the microscopic domains $\oeo$ and $\oet$, as well as the interface $\ge$, and explain some geometrical properties. Then we state the microscopic equation for given $\epsilon$ and give the assumptions on the data.

\subsection{The microscopic geometry}

Let $\Omega \subset \R^n$ with Lipschitz-boundary and  $\epsilon>0 $  a sequence with $\epsilon^{-1} \in \N$. We define the unit cube $Y:=(0,1)^n$ and $Y_2 \subset Y$ such that $\overline{Y_2} \subset Y$, so $Y_2 $ strictly included in $Y$. Further, we define $Y_1:= Y\setminus \overline{Y_2}$ and $\Gamma := \partial Y_2$, and we suppose that $\Gamma \in C^{1,1}$. We assume that $Y_1$ is connected and for the sake of simplicity we also assume that $Y_2$ is connected. The general case of disconnected $Y_2$ is easily obtained by considering the connected components of $Y_2$, see also Remark \ref{RemarkEnd}. Now, the microscopic domains $\oeo$ and $\oet$ are defined by scaled and shifted reference elements $Y_j$ for $j=1,2$.  Let $K_{\epsilon}:= \{k\in \Z^n \, : \, \epsilon(k + Y) \subset \Omega\}$ and define
\begin{align*}
\oet:= \bigcup_{k\in K_{\epsilon}} \epsilon (Y_2 + k), \quad\quad \oeo := \Omega \setminus \overline{\oet}, \quad\quad \ge := \partial \oet.
\end{align*}
Hence, $\ge $ denotes  the oscillating interface between $\oeo$ and $\oet$. Due to the assumptions on $Y_1$ and $Y_2$ it holds that $\oeo$ is connected and $\oet$ is disconnected, and $\ge \in C^{1,1}$ is not touching the outer boundary $\partial \Omega$.

\subsection{The microscopic model}

We are looking for a solution $(\uepso,\uepst)$ with $\uepsj: (0,T)\times \oej \rightarrow \R$ for $j=1,2$, such that  it holds that
\begin{align}
\begin{aligned}
\label{MicroscopicModel}
\partial_t \uepsj - \nabla \cdot \big(\depsj \nabla \uepsj \big) &= \fepsj(\uepsj) &\mbox{ in }& (0,T)\times \oej,
\\
\epsilon \big( \partial_t \uepsj - \ngeps \cdot \big( \dgepsj \ngeps \uepsj\big) - \hepsj(\uepso , \uepst) \big) &= - \depsj \nabla \uepsj\cdot \nu &\mbox{ on }& (0,T) \times \ge,
\\
-\depsj \nabla \uepsj \cdot \nu &= 0 &\mbox{ on }& (0,T)\times \partial \Omega,
\\
\uepsj(0) &= u_{\epsilon,i}^j &\mbox{ in }& \oej,
\\
\uepsj|_{\ge}(0) &= u_{\epsilon,i,\ge}^j &\mbox{ on }& \ge,
\end{aligned}
\end{align}
where $\nu$ denotes the outer unit normal (we neglect a subscript for the underlying domain, since this should be clear from the context), and $\uepsj|_{\ge}$ denotes the trace of $\uepsj$ on $\ge$.  If it is clear from the context, we use the same notation for a function and its trace, for example we just write $\uepsj $ for $\uepsj|_{\ge}$.
The precise weak formulation of the micro-model above is stated in Section \ref{SectionExistence}, see $\eqref{VariationalEquationMicroscopicProblem}$, after introducing the necessary function spaces.

In the following, with $T_y \Gamma$ and $T_x \ge$ for $y \in \Gamma $ and $x \in \ge$ we denote the tangent spaces of $\Gamma$ at $y$ and $\ge$ at $x$, respectively. The orthogonal projection $P_{\Gamma}(y): \R^n \rightarrow T_y \Gamma $ for $y \in \Gamma$ is given by
\begin{align*}
P_{\Gamma}(y) \xi = \xi - (\xi\cdot \nu(y))\nu(y) \quad\quad \mbox{for } \xi \in \R^n,
\end{align*}
where $\nu(y) $ denotes the outer unit normal at $y\in \Gamma$. Let us extend the unit normal $Y$-periodically. Then, the orthogonal projection $P_{\ge}(x): \R^n \rightarrow T_x \ge $ for $x \in \ge$ is given by
\begin{align*}
P_{\ge}(x)\xi = \xi - \left(\xi \cdot \nu\left(\fxe\right)\right)\nu\left(\fxe\right) \quad \quad \mbox{for } \xi \in \R^n.
\end{align*} 

\textbf{Assumptions on the data:}
\\
In the following let $j \in \{1,2\}$.
\begin{enumerate}
[label = (A\arabic*)]
\item\label{AssumptionDBulk} For the bulk-diffusion we have $\depsj(x):= D^j\left(\fxe\right)$ with $D^j \in L^{\infty}_{\per}(Y_j)^{n\times n}$ symmetric and coercive, \ie there exits $c_0 >0$ such that for almost every $y \in Y_j$
\begin{align*}
D^j(y)\xi \cdot \xi \geq c_0 |\xi|^2 \quad \mbox{ for all } \xi \in \R^n.
\end{align*}
\item\label{AssumptionDSurface} For the surface-diffusion we suppose $\dgepsj(x):= D_{\Gamma}^j\left(\fxe\right) $ with $D_{\Gamma}^j \in  L^{\infty}_{\per}(\Gamma)^{n\times n}$ symmetric and $D_{\Gamma}^j(y)|_{T_y\Gamma}: T_y \Gamma \rightarrow T_y \Gamma$ for almost every $y \in \Gamma$. Further, we assume that $D_{\Gamma}^j$ is coercive, \ie there exists $c_0>0$ such that for almost every $y \in \Gamma$
\begin{align*}
D_{\Gamma}^j(y)\xi\cdot \xi \geq c_0 |\xi|^2 \quad\mbox{ for all } \xi \in T_y \Gamma.
\end{align*}
\item\label{Assumptionfj} For the reaction-rates in the bulk domains we suppose that $\fepsj (t,x,z):= f^j\left(t,\fxe,z\right)$ with $f^j \in L^{\infty}((0,T)\times Y_j \times \R)$ is $Y$-periodic with respect to the second variable and uniformly Lipschitz continuous with respect to the last variable, \ie there exists $C>0$ such that for all $z,w \in \R$ and almost every $(t,y) \in (0,T)\times Y_j$ it holds that
\begin{align*}
|f^j(t,y,z)-f^j(t,y,w)| \le C |z - w|.
\end{align*} 
\item\label{Assumptionhj} For the reaction-rates on the surface we suppose that $\hepsj(t,x,z_1,z_2) := h^j\left(t,\fxe,z_1,z_2\right)$ with $h^j \in L^{\infty}((0,T)\times \Gamma \times \R^2)$ is $Y$-periodic with respect to the second variable and uniformly Lipschitz continuous with respect to the last variable, \ie there exists $C>0$ such that for all $z_1,z_2,w_1,w_2 \in \R$ and almost every $(t,y)\in (0,T)\times \Gamma $ it holds that
\begin{align*}
|h^j(t,y,z_1,z_2)-h^j(t,y,w_1,w_2)| \le C \big(|z_1 - w_1| + |z_2 - w_2|\big).
\end{align*}
\item\label{AssumptionInitialConditions} For the initial conditions we assume $u_{\epsilon,i}^j \in L^2(\oej)$ and $u_{\epsilon,i,\ge}^j \in L^2(\ge)$ and there exists $u_{0,i}^j \in L^2(\Omega) $ and $u_{0,i,\Gamma}^j \in L^2(\Omega)$ such that 
\begin{align*}
u_{\epsilon,i}^j &\rightarrow u_{0,i}^j &\mbox{ in the two-scale sense}&,
\\
u_{\epsilon,i,\ge}^j &\rightarrow u_{0,i,\Gamma}^j &\mbox{ in the two-scale sense on } \ge &,
\end{align*}
and it holds that
\begin{align*}
\sqrt{\epsilon}\|u_{\epsilon,i,\ge}^j\|_{L^2(\ge)} \le C.
\end{align*}
Additionally we assume that the sequences $u_{\epsilon,i}^2$ and $u_{\epsilon,i,\ge}^2$   converge strongly in the two-scale sense.
\end{enumerate}
For the definition of the two-scale convergence see Section \ref{SectionTwoScaleConvergenceUnfolding}. We emphasize that due to the convergences in \ref{AssumptionInitialConditions} it holds that
\begin{align*}
\|u_{\epsilon,i}^j\|_{L^2(\oej)} + \sqrt{\epsilon}\|u_{\epsilon,i,\ge}^j\|_{L^2(\ge)} \le C.
\end{align*}

\section{Existence of a weak solution and \textit{a priori} estimates}
\label{SectionExistence}

The aim of this section is the investigation of the microscopic problem $\eqref{MicroscopicModel}$. We introduce appropriate function spaces and show existence and uniqueness of a microscopic solution. Further, we derive \textit{a priori} estimates for the solution depending explicitly on $\epsilon$. These estimates form the basis for the derivation of the macroscopic problem $\eqref{MacroscopicModel}$ by using the compactness results from Section \ref{SectionTwoScaleConvergenceUnfolding}.

\subsection{Function spaces}
\label{SectionFunctionsSpaces}


Due to the Laplace-Beltrami operator in the boundary condition in $\eqref{MicroscopicModel}$, it is not enough to consider the usual Sobolev space $H^1(\oej)$ as a solution space, because we need more regular traces. This gives rise to deal with the following function spaces for $j=1,2$:
\begin{align}
\begin{aligned}
\label{DefinitionHilbertraum}
\Hjeps&:= \left\{\pepsj \in H^1(\oej)\, : \, \pepsj|_{\ge} \in H^1(\ge)\right\}, 
\\
\Hj &:= \left\{ \phi \in H^1(Y_j) \, : \, \phi|_{\Gamma} \in H^1(\Gamma)\right\},
\end{aligned}
\end{align} 
with the inner products
\begin{align*}
 (\pepsj,\psi_{\epsilon}^j)_{\Hjeps} &:= (\pepsj,\psi_{\epsilon}^j)_{H^1(\oej)} + \epsilon (\pepsj,\psi_{\epsilon}^j)_{H^1(\ge)},
 \\
 (\phi,\psi)_{\Hj} &:= (\phi,\psi)_{H^1(Y_j)} + (\phi,\psi)_{H^1(\Gamma)}.
\end{align*}
The associated norms are denoted by $\|\cdot\|_{\Hjeps}$ and $\|\cdot\|_{\Hj}$. Obviously, the spaces $\Hjeps$ and $\Hj$ are separable Hilbert spaces and we have the dense embeddings $C^{\infty}(\overline{\oej}) \subset \Hjeps$ and $C^{\infty}(\overline{Y_j})\subset \Hj$, see \cite[Proposition 5]{Gahn2019}. We also define the space 
\begin{align*}
\Ljeps := L^2(\oej) \times L^2(\ge), \quad \quad \Lj:= L^2(Y_j)\times L^2(\Gamma)
\end{align*}
with inner products 
\begin{align*}
(\pepsj,\psi_{\epsilon}^j)_{\Ljeps} &:= (\pepsj,\psi_{\epsilon}^j)_{L^2(\oej)} + \epsilon (\pepsj,\psi_{\epsilon}^j)_{L^2(\ge)},
 \\
 (\phi,\psi)_{\Lj} &:= (\phi,\psi)_{L^2(Y_j)} + (\phi,\psi)_{L^2(\Gamma)},
\end{align*}
and again denote the associated norms by $\|\cdot\|_{\Ljeps}$ and $\|\cdot \|_{\Lj}$.
Obviously, we have
\begin{align*}
\Hjeps \overset{\sim}{=} \left\{ (\ueps,\veps) \in H^1(\oej)\times H^1(\ge) \, : \, \ueps|_{\ge} = \veps\right\},
\end{align*}
and a similar result for $\Hj$. Therefore, we have the following Gelfand-triples:
\begin{align*}
\Hjeps \hookrightarrow \Ljeps \hookrightarrow \Hjeps',\quad \quad
\Hj \hookrightarrow \Lj \hookrightarrow \Hj.
\end{align*}
We will also make use for $\alpha \in \left(\frac12,1\right]$ of the function space
\begin{align*}
\Hj^{\alpha}:= \left\{\phi \in H^{\alpha}(Y_j) \, : \, \phi|_{\Gamma} \in H^{\alpha}(\Gamma)\right\}
\end{align*}
with inner product
\begin{align*}
(\phi,\psi)_{\Hj^{\alpha}} :=  (\phi,\psi)_{H^{\alpha}(Y_j)} + (\phi,\psi)_{H^{\alpha}(\Gamma)}.
\end{align*}
By definition we have $\Hj = \Hj^1$. Obviously, we have the compact embedding $\Hj \hookrightarrow \Hj^{\alpha}$ for $\alpha \in \left(\frac12,1\right)$.

\subsection{Existence and uniqueness of a weak solution}

A weak solution of Problem $\eqref{MicroscopicModel}$ is defined in the following way: The tuple $(\uepso,\uepst)$ is a weak solution of $\eqref{MicroscopicModel}$ if for $j=1,2$
\begin{align*}
\uepsj \in L^2((0,T),\Hjeps)\cap H^1((0,T),\Hjeps'),
\end{align*}
$\uepsj$ and $\uepsj|_{\ge}$ fulfill the initial condition $\uepsj(0) = u_{\epsilon,i}^j$ and $\uepsj|_{\ge}(0) = u_{\epsilon,i,\ge}^j$, and for all $\pepsj\in  \Hjeps$ it holds almost everywhere in $(0,T)$
\begin{align}
\begin{aligned}
\label{VariationalEquationMicroscopicProblem}
\langle \partial_t \uepsj , \pepsj&\rangle_{\Hjeps',\Hjeps}  + (\depsj \nabla \uepsj , \nabla \pepsj )_{\oej} + \epsilon (\dgepsj \ngeps \uepsj , \ngeps \pepsj )_{\ge}
\\
&= (\fepsj (\uepsj) , \pepsj)_{\oej}  + \epsilon (\hepsj(\uepso,\uepst),\pepsj)_{\ge}.
\end{aligned}
\end{align}
Here $(\cdot,\cdot)_U$ stands for the inner product on $L^2(U)$, for a suitable set $U \subset \R^n$, and for a Banach space $X$ and its dual $X'$ we write $\langle \cdot , \cdot \rangle_{X',X}$ for the duality pairing between $X'$ and $X$. The scaling factor $\epsilon$ for the time-derivative in $\eqref{MicroscopicModel}$ is included in the duality pairing $\langle \cdot,\cdot \rangle_{\Hjeps',\Hjeps}$. In fact, if additionally it holds that $\partial_t \uepsj \in L^2((0,T),H^1(\oej)')$ and $\partial_t \uepsj|_{\ge} \in L^2((0,T),H^1(\ge)')$ with respect to the Gelfand-triples $H^1(\oej) \hookrightarrow L^2(\oej) \hookrightarrow H^1(\oej)'$ and $H^1(\ge) \hookrightarrow L^2(\ge) \hookrightarrow H^1(\ge)'$, we get for all $\pepsj \in \Hjeps$
\begin{align*}
\langle \partial_t \uepsj , \pepsj \rangle_{\Hjeps',\Hjeps} = \langle \partial_t \uepsj, \pepsj\rangle_{H^1(\oej)',H^1(\oej)} + \epsilon \langle \partial_t \uepsj|_{\ge} , \pepsj|_{\ge} \rangle_{H^1(\ge)',H^1(\ge)}.
\end{align*}

\begin{proposition}\label{PropositionExistenceUniqunessMicroProblem}
There exists a unique weak solution $\ueps = (\uepso,\uepst)$ of the microscopic problem $\eqref{MicroscopicModel}$.
\end{proposition}
\begin{proof}
This is an easy consequence of the Galerkin-method and the Leray-Schauder principle, where we have to use similar estimates as in Proposition \ref{PropositionAprioriEstimates} below. The uniqueness follows from standard energy estimates. For more details see \cite{GahnDissertation}.
\end{proof}

\subsection{\textit{A priori} estimates}
\label{SectionAprioriEstimates}

We derive \textit{a priori} estimates for the microscopic solution depending explicitly on $\epsilon$. These estimates are necessary for the application of the two-scale compactness results from Section \ref{SectionTwoScaleConvergenceUnfolding} to derive the macroscopic model. In a first step, we give estimates in the spaces $L^2((0,T),\Hjeps)$ and $H^1((0,T),\Hjeps')$. Such kind of estimates are also needed to establish the existence of a weak solution via the Galerkin-method. In a second step, we derive estimates for the difference of shifted microscopic solution with respect to the macroscopic variable. These estimates are necessary for strong two-scale compactness results in the disconnected domain.

The following trace inequality for perforated domains will be used frequently throughout the paper and follows easily by a standard decomposition argument and the trace inequality on the reference element $Y_j$, see also \cite[Theorem II.4.1 and Exercise II.4.1]{Galdi}: For every $\theta >0$ there exists a $C(\theta)>0$ such that  for every $\veps \in H^1(\oej)$ it holds that
\begin{align}\label{TraceInequality}
\|\veps\|_{L^2(\ge)} \le \frac{C(\theta)}{\sqrt{\epsilon}} \|\veps\|_{L^2(\oej)} + \theta \sqrt{\epsilon} \|\nabla \veps \|_{L^2(\oej)}. 
\end{align}
%

\begin{proposition}\label{PropositionAprioriEstimates}
The weak solution $\ueps = (\uepso,\uepst)$ of the microscopic problem $\eqref{MicroscopicModel}$ fulfills the following \textit{a priori} estimate
\begin{align*}
\|\partial_t \uepsj\|_{L^2((0,T),\Hjeps')} + \|\uepsj\|_{L^2((0,T),\Hjeps)} \le C,
\end{align*}
for a constant $C > 0$ independent of $\epsilon$.
\end{proposition}
\begin{proof}
We choose $\uepsj$ as a test-function in $\eqref{SectionExistence}$ (for $j=1,2$) to obtain with the Assumptions \ref{Assumptionfj} and  \ref{Assumptionhj} on $f^j$ and $h^j$
\begin{align*}
\frac12 \frac{d}{dt} \Vert \uepsj  \Vert_{\Ljeps}^2 +  (\depsj & \nabla \uepsj , \nabla \uepsj  )_{\oej} + \epsilon (\dgepsj \ngeps \uepsj , \ngeps \uepsj )_{\ge}
\\
&= (\fepsj (\uepsj) , \uepsj)_{\oej}  + \epsilon (\hepsj(\uepso,\uepst),\uepsj)_{\ge}
\\
&\le C\left(1 + \Vert \uepsj\Vert_{L^2(\oej)}^2 + \epsilon\Vert \uepso\Vert^2_{L^2(\ge)} + \epsilon \Vert \uepst\Vert^2_{L^2(\ge)} \right)
\\
&\le C\left(1 + \Vert \uepso\Vert^2_{\Loeps} + \Vert \uepst\Vert^2_{\Lteps}\right).
\end{align*} 
Using the coercivity of $\depsj$ and $\dgepsj$ from the Assumptions \ref{AssumptionDBulk} and \ref{AssumptionDSurface}, we obtain for $j=1,2$
\begin{align*}
\frac{d}{dt}\Vert \uepsj\Vert_{\Ljeps}^2 + \Vert \nabla \uepsj\Vert_{L^2(\oej)}^2 + \epsilon \Vert \ngeps \uepsj \Vert_{L^2(\ge)}^2 \le C\left(1 + \Vert \uepso\Vert^2_{\Loeps} + \Vert \uepst\Vert^2_{\Lteps}\right).
\end{align*}
Summing over $j=1,2$, integrating with respect to time, the Assumption \ref{AssumptionInitialConditions}, and the Gronwall-inequality implies the boundedness of $\Vert \uepsj\Vert_{L^2((0,T),\Hjeps)}$ uniformly with respect to $\epsilon$.

It remains to check the   bound for the time-derivative $\partial_t \uepsj$. As a test-function in $\eqref{VariationalEquationMicroscopicProblem}$ we choose $\pepsj \in \Hjeps$ with $\Vert \pepsj \Vert_{\Hjeps} \le 1$ to obtain (using the boundedness of the diffusion tensors and again the growth condition for $h^j$ and $f^j$):
\begin{align*}
\langle \partial_t \uepsj , \pepsj \rangle_{\Hjeps',\Hjeps} \le C\left( \Vert \uepso \Vert_{\Hoeps} + \Vert \uepst \Vert_{\Hteps}  \right)\Vert \pepsj\Vert_{\Hjeps}   
\le  C\left( \Vert \uepso \Vert_{\Hoeps} + \Vert \uepst \Vert_{\Hteps}  \right).
\end{align*}
Squaring, integrating with respect to time, and the boundedness of $\uepsj$ for $j=1,2$ already obtained above implies the desired result.
\end{proof}

Next, we derive estimates for the difference of the shifted functions. 
First of all, we introduce some additional notations. 
For $h>0$ let us define 
\begin{align*}
\Omega_h &:= \{x\in \Omega \, : \, \mathrm{dist}(x,\partial \Omega) > h \} ,
\\
K_{\epsilon,h} &:= \{k\in \Z^n\, : \, \epsilon(Y + k) \subset \Omega_h\},
\\
\Omega_{\epsilon,h} &:= \mathrm{int} \bigcup_{k \in K_{\epsilon,h}} \epsilon \big(\overline{Y} + k \big),
\end{align*}
and the related perforated domains and the related surface
\begin{align*}
\oeht&:=  \bigcup_{k \in K_{\epsilon,h}} \epsilon\big(Y_2 + k \big), \quad \oeho := \Omega_{\epsilon,h} \setminus \overline{\oeht},\quad
\geh := \partial \oeht.
\end{align*}
For $l \in \Z^n $ with $|l\epsilon|< h$ and $G_{\epsilon,h} \in \{\Omega_{\epsilon,h} , \oeho,\oeht\}$, we define for an arbitrary function $\veps : G_{\epsilon,h} \rightarrow \R$ the shifted function
\begin{align*}
\veps^l(x):= \veps(x + l\epsilon),
\end{align*} 
and the difference between the shifted function and the function itself
\begin{align} \label{DefinitionDifference}
\delta_l \veps(x):=\delta \veps(x):= \veps^l(x) - \veps(x) = \veps (x+l\epsilon) - \veps(x).
\end{align}
Here, in the writing $\delta \veps$ we neglect the dependence on $l$ if it is clear from the context. Further, we define $\Hjepsh$ in the same way as $\Hjeps$ in $\eqref{DefinitionHilbertraum}$ by replacing $\oej$ and $\ge$ with $\oehj$ and $\geh$. In the same way we define $\Ljepsh$. 
Further, for any function $\pepsh\in \Htepsh$ we write $\bpepsh $ for the zero extension to $\oet$. Especially it holds that $\bpepsh \in \Hteps$, since $\oet$ is disconnected.


\begin{proposition}\label{PropositionEstimateShifts}
Let $0 <h \ll 1$, then for all $l \in \Z^n$  with $|\epsilon l|<h$, it holds that
\begin{align*}
\|\delta \uepst\|_{L^{\infty}((0,T),\Ltepsh)} + \|&\nabla \delta \uepst \|_{L^2((0,T),\Ltepsh)} 
\\
&\le C \left( \|\delta \uepso \|_{L^2((0,T),L^2(\oeho))} + \big\|\delta \big(u_{\epsilon,i}^2,u_{\epsilon,i,\ge}^2\big)\big\|_{\Ltepsh} + \epsilon  \right),
\end{align*}
for a constant $C>0$ independent of $h,\, \epsilon,$ and $l$.
\end{proposition}

\begin{proof}
Let $0 <h \ll 1$ and $l \in \Z^n$  with $|\epsilon l|<h$, and we shortly write $\ueps^{2,l} := \left(\uepst\vert_{\oeht}\right)^l$, \ie the shifts with respect to $l\epsilon$ of the restriction $\uepst\vert_{\oeht}$ (we neglect the index $h$). In the same way we define $\ueps^{1,l}$. 
Let $\pepsh \in \Htepsh$. Then, for $x \in \oet \setminus (\oeht + \epsilon l) $ it holds that $x - l\epsilon \notin \oeht$ and therefore $\bpepsh^{-l}(x) = 0$ and similar from $x \in \ge \setminus(\geh + \epsilon l)$ it follows $\bpepsh^{-l}(x) = 0$ . This implies for all $\psi \in C_0^{\infty}(0,T)$
\begin{align*}
\int_0^T &\left(\ueps^{2,l},\pepsh\right)_{\Ltepsh} \psi'(t) dt 
\\
&= \int_0^T \left[ \int_{\oeht} \uepst (t, x + l\epsilon) \pepsh(x) dx + \epsilon \int_{\geh} \uepst(t,x+l\epsilon) \pepsh(x) d\sigma\right] \psi'(t) dt
\\
&= \int_0^T \left[ \int_{\oeht + l \epsilon  } \uepst(t,x) \pepsh^{-l}(x) dx + \epsilon \int_{\geh + l \epsilon} \uepst(t,x) \pepsh^{-l}(x) d\sigma\right] \psi'(t) dt
\\
&= \int_0^T \left[ \int_{\oet  } \uepst(t,x) \pepsh^{-l}(x) dx + \epsilon \int_{\ge } \uepst(t,x) \pepsh^{-l}(x) d\sigma\right] \psi'(t) dt
\\
&= \int_0^T \left(\uepst , \bpepsh^{-l} \right)_{\Lteps} \psi'(t) dt = - \int_0^T \left\langle \partial_t \uepst , \bpepsh^{-l} \right\rangle_{\Hteps',\Hteps} \psi(t) dt .
\end{align*}
Hence, we have $\partial_t \ueps^{2,l} \in L^2((0,T),\Htepsh')$ with
\begin{align*}
\big\langle \partial_t \ueps^{2,l} , \pepsh \big\rangle_{\Htepsh',\Htepsh} = \big\langle \partial_t \uepst , \bpepsh^{-l} \big\rangle_{\Hteps',\Hteps}
\end{align*}
almost everywhere in $(0,T)$. Using $\bpepsh^{-l} \in \Hteps$ as a test-function in $\eqref{VariationalEquationMicroscopicProblem}$, we obtain using the periodicity of $D^2$, $D^2_{\Gamma}$, $f^2$, and $h^2$, by an elemental calculation
\begin{align*}
\big\langle \partial_t \uepst , \bpepsh^{-l} \big\rangle_{\Hteps',\Hteps} =& -\big(D_{\epsilon}^2 \nabla \ueps^{2,l} , \nabla \pepsh \big)_{\oeht} - \epsilon \big(D_{\ge}^2 \ngeps \ueps^{2,l} , \ngeps \pepsh\big)_{\geh}
\\
&+ \big(f_{\epsilon}^2(\ueps^{2,l}) , \pepsh \big)_{\oeht} + \epsilon \big(h_{\epsilon}^2(\ueps^{1,l},\ueps^{2,l}) , \pepsh \big)_{\geh}.
\end{align*}
Subtracting the above equation for $l=0$ and arbitrary $l\in \Z^n$ with $|\epsilon l|<h$ we obtain
\begin{align*}
\langle \partial_t& \delta \uepst,\pepsh\rangle_{\Htepsh',\Htepsh} + \big(D_{\epsilon}^2 \nabla \delta  \uepst, \nabla \pepsh\big)_{\oeht} + \epsilon \big(D_{\ge}^2 \ngeps \delta \uepst , \ngeps \pepsh \big)_{\geh} 
\\
&= \big(f_{\epsilon}^2\big((\uepst)^l\big) - f_{\epsilon}^2(\uepst),\pepsh\big)_{\oeht} + \epsilon \big(h_{\epsilon}^2\big((\uepso)^l,(\uepst)^l\big) - h_{\epsilon}^2(\uepso,\uepst) , \pepsh\big)_{\geh}.
\end{align*}
Choosing $\pepsh:=\delta \uepst$ (more precisely we take the restriction of $\uepst$ to $\oeht$) we obtain with the coercivity of $D^2$ and $D^2_{\Gamma}$, as well as the Lipschitz continuity of $f^2$ and $h^2$
\begin{align*}
\frac12 \frac{d}{dt} \|\delta \uepst \|^2_{\Ltepsh} + c_0 \|\nabla \delta \uepst\|_{\Ltepsh}^2 &\le C \left( \|\delta \uepst \|^2_{L^2(\oeht)} + \epsilon \sum_{j=1}^2 \|\delta \uepsj \|^2_{L^2(\geh)}\right)
\\
&\le  \sum_{j=1}^2 \left( C(\theta)\Vert \delta \uepsj \Vert_{L^2(\oehj)}^2 + \epsilon \theta \Vert \nabla \uepsj \Vert^2_{L^2(\oehj)} \right),
\end{align*} 
for arbitrary $\theta >0$, where we used the trace-inequality $\eqref{TraceInequality}$. Choosing $\theta$ small enough the gradient term for $j=2$ can be absorbed from the left-hand side. Integrating with respect to time, using the \textit{a priori} estimates from Proposition \ref{PropositionAprioriEstimates} for the gradients of $\uepso$, as well as the Gronwall-inequality, we obtain the desired result.
\end{proof}

\section{Two-scale compactness results}
\label{SectionTwoScaleConvergenceUnfolding}

In this section we prove general strong two-scale compactness results for functions in the space $L^2((0,T),\Hjeps) \cap H^1((0,T),\Hjeps')$ for $j\in \{1,2\}$ based on suitable \textit{a priori} estimates. These estimates are fulfilled by the microscopic solution $\ueps = (\uepso,\uepst)$ which fulfills Proposition \ref{PropositionAprioriEstimates} and Proposition \ref{PropositionEstimateShifts}, but are not restricted to them.
The connected and disconnected case are completely different and are therefore treated separately. These strong compactness results are enough to pass to the limit in the nonlinear terms in the microscopic equation $\eqref{VariationalEquationMicroscopicProblem}$, in fact we have:
\begin{lemma}\label{LemmaStrongTSSurface} 
Let $p \in (1,\infty)$.
\begin{enumerate}
[label = (\roman*)]
\item\label{LemmaStrongTSSurfaceBulk} For $j \in \{1,2\}$ let $(\uepsj) \subset L^p((0,T)\times \oej)$ be a  sequence converging strongly in the two-scale sense to $u_0^j \in L^p((0,T)\times \Omega \times Y_j)$. Further $f:[0,T]\times Y_j \times \R \rightarrow \R$ is continuous, $Y$-periodic with respect to the second  variable, and fulfills the growth condition
\begin{align*}
\vert f(t,y,z)\vert \le C(1  + \vert z \vert ) \quad \mbox{ for all } (t,y,z) \in [0,T]\times Y_j \times \R.
\end{align*}
Then it holds up to a subsequence
\begin{align*}
f\left(\cdot_t,\frac{\cdot_x}{\epsilon} ,\uepsj \right) &\rightarrow f\left(\cdot_t,\cdot_y , u_0^j\right) &\mbox{ in the two-scale sense in } L^p.
\end{align*}
\item\label{LemmaStrongTSSurfaceSurface} Let $(\ueps) \subset L^p((0,T)\times \ge)$ be a  sequence converging strongly in the two-scale sense on $\ge$ to $u_0 \in L^p((0,T)\times \Omega \times \Gamma)$. Further $h:[0,T]\times \Gamma \times \R \rightarrow \R$ is continuous, $Y$-periodic with respect to the second  variable, and fulfills the growth condition
\begin{align*}
\vert h(t,y,z)\vert \le C(1  + \vert z \vert ) \quad \mbox{ for all } (t,y,z) \in [0,T]\times \Gamma \times \R.
\end{align*}
Then it holds up to a subsequence
\begin{align*}
h\left(\cdot_t,\frac{\cdot_x}{\epsilon} ,\ueps \right) &\rightarrow h(\cdot_t,\cdot_y , u_0) &\mbox{ in the two-scale sense on } \ge \mbox{ in } L^p.
\end{align*}
\end{enumerate}
We emphasize that for functions $f$ and $h$ uniformly Lipschitz-continuous with respect to the last variable, the growth conditions are fulfilled. For such Lipschitz-continuous functions we also easily obtain the strong two-scale convergence of the whole sequence. 
\end{lemma}

\begin{proof}
We only prove \ref{LemmaStrongTSSurfaceSurface}. The other statement follows the same way. Due to Lemma \ref{PropAequivalenzTSKonvergenzUnfolding} the sequence $\te \ueps$ converges in $L^p((0,T)\times \Omega \times \Gamma)$ to $u_0$. Hence, up to a subsequence, $\te \ueps \rightarrow u_0$ almost everywhere in $(0,T)\times \Omega \times \Gamma$. Further we have
\begin{align*}
\te \left( h\left(\cdot_t,\frac{\cdot_x}{\epsilon} ,\ueps \right) \right) = h\left(\cdot_t , y , \te \ueps \right) \rightarrow h(\cdot_t,\cdot_y,u_0) \quad \mbox{ a.e. in } (0,T)\times \Omega \times \Gamma.
\end{align*}
The growth condition on $h$ implies  $\te \left( h\left(\cdot_t,\frac{\cdot_x}{\epsilon} ,\ueps \right) \right)$ bounded in $ L^p((0,T)\times \Omega \times \Gamma)$. Egorov's theorem (see also \cite[Theorem 13.44]{HewittStromberg}) implies
\begin{align*}
\te \left( h\left(\cdot_t,\frac{\cdot_x}{\epsilon} ,\ueps \right) \right) \rightharpoonup h(\cdot_t ,\cdot_y , u_0) \quad \mbox{ weakly in } L^p((0,T)\times \Omega \times \Gamma).
\end{align*}
\end{proof}

As a direct consequence we obtain by density:

\begin{lemma}\label{LemmaStrongTwoScaleLpSenseNonlinear}
Let $f$ be as in Lemma \ref{LemmaStrongTSSurface}. Further, let $\{\veps\}$ be a bounded sequence in $L^2((0,T)\times \oej)$ which converges strongly in the $L^p$-two-scale sense to $v_0 \in L^2((0,T)\times \Omega \times Y)$ for $p \in [1,2)$. Then it holds that
\begin{align*}
f\left(\cdot_t,\frac{\cdot_x}{\epsilon} ,\veps(\cdot_x)\right) &\rightarrow f(\cdot_t,\cdot_y , v_0 (\cdot_x)) &\mbox{ in the two-scale sense in }L^2.
\end{align*}
A similar result holds on the oscillating surface $\ge$.
\end{lemma}

\subsection{The connected domain $\oeo$}
\label{SectionConnectedDomain}
Here we give a strong compactness result for a sequence in the connected domain $\oeo$ under suitable \textit{a priori} estimates. The case of a connected perforated domain can be treated more easily than a disconnected domain, because we can extend a bounded sequence in $H^1(\oe)$ to a bounded sequence in $H^1(\Omega)$, due to \cite{Acerbi1992,CioranescuSJPaulin}. Hence, we can work in fixed Bochner spaces (not depending on $\epsilon$) and use standard methods from functional analysis.
For this we need control for the time-variable, which can be obtained from the uniform bound of the time-derivative $\partial_t \uepso$. However, since $\partial_t \uepso$ is pointwise only an element in the space $\Hoeps'$ it is not clear if the time-derivative of the extension of $\uepso$ exists and if it is bounded uniformly with respect to $\epsilon$ (Unfortunately, this circumstance is often overseen in the existing literature).
The following Lemma gives us an estimate for the difference of the shifts with respect to time for functions with generalized time-derivative. It is just an easy generalization of \cite[Lemma 9]{Gahn}.

\begin{lemma}\label{LemmaNikolskiiEmbedding}
Let $V$ and $H$ be Hilbert spaces   and we assume that $(V,H,V^{\prime})$ is a Gelfand-triple. Let $v \in L^2((0,T),V) \cap H^1((0,T),V^{\prime})$. Then, for every $\phi \in V$ and almost every $t \in (0,T)$, $s\in(-T,T)$, such that $t + s \in (0,T)$, we have
\begin{align*}
\big|(v(t+s) - v(t), \phi)_H\big| \le \sqrt{|s|} \|\phi\|_V \|\partial_t v \|_{L^2((t,t+s),V^{\prime})}.
\end{align*}
Especially, it holds that
\begin{align*}
\big\|v(t+s) - v(t) \big\|^2_H \le \sqrt{|s|} \big\|v(t+s) - v(t)\big\|_V \|\partial_t v\|_{L^2((t,t+s),V^{\prime})}.
\end{align*}
\end{lemma}
\begin{proof}
The proof follows the same lines as the proof of \cite[Lemma 9]{Gahn}, if we replace  the Gelfand-triple $(H^1(\Omega_j^{\epsilon}),L^2(\Omega_j^{\epsilon}),H^1(\Omega_j^{\epsilon})')$ by the Gelfand-triple $(V,H,V')$.
\end{proof}

In the following, for $\veps \in H^1(\oeo)$ we denote by $\tveps \in H^1(\Omega)$ the extension from \cite{Acerbi1992,CioranescuSJPaulin} with 
\begin{align*}
\Vert \tveps \Vert_{L^2(\Omega)} \le C \Vert \veps\Vert_{L^2(\oeo)}, \quad \Vert \nabla \tveps \Vert_{L^2(\Omega)} \le C \Vert \nabla \veps \Vert_{L^2(\oeo)},
\end{align*}
with a constant $C>0$ independent of $\epsilon$. 

\begin{proposition}\label{StrongConvergenceConnectedDomain}
Let $(\veps)\subset L^2((0,T),\Hoeps)\cap H^1((0,T),\Hoeps')$ be a sequence with
\begin{align}\label{PropStarkeKonvergenzZshgGebiet}
\Vert \partial_t \veps \Vert_{L^2((0,T),\Hoeps')} + \Vert \veps \Vert_{L^2((0,T),\Hoeps)} \le C.
\end{align}
There exists $v_0 \in L^2((0,T),H^1(\Omega))$ such that for all $\beta \in \left(\frac12, 1\right)$  up to a subsequence it holds that
\begin{align*}
\tveps  \rightarrow v_0 \quad \mbox{ in } L^2((0,T),H^{\beta}(\Omega)).\end{align*}
Further, it holds that (up to a subsequence)
\begin{align*}
\te \veps \rightarrow v_0 \quad \mbox{ in } L^2((0,T)\times \Omega , \Ho).
\end{align*}
\end{proposition}
\begin{proof}
Since $\tveps$ is bounded in $L^2((0,T),H^1(\Omega))$ there exists $v_0 \in L^2((0,T),H^1(\Omega))$, such that up to a subsequence $\veps $ converges weakly to $v_0$ in $L^2((0,T),H^1(\Omega))$.  Lemma \ref{LemmaNikolskiiEmbedding} and inequality $\eqref{PropStarkeKonvergenzZshgGebiet}$ imply for  $0<h\ll 0$
\begin{align}
\begin{aligned}\label{AuxiliaryEstimateShiftTime}
\int_0^{T-h}\|\veps (t + h ) - \veps\|^2_{\Loeps} dt
&\le  C \sqrt{h}  \|\partial_t \veps\|_{L^2((0,T),\Hoeps')} \int_0^{T-h} \|\veps (t+h) - \veps\|_{\Hoeps} dt 
\\
&\le C\sqrt{h}.
\end{aligned}
\end{align}
Now, from the properties of the extension $\tveps$ we obtain
\begin{align*}
\int_0^{T-h} \Vert \tveps(t+h) - \tveps\Vert^2_{L^1(\Omega)} dt \le C \int_0^{T-h}\Vert \veps(t+h) - \veps\Vert^2_{L^2(\oeo)} dt \le C\sqrt{h}.
\end{align*}
Since $H^1(\Omega) \hookrightarrow H^{\beta}(\Omega)$ is compact for $\beta \in \left(\frac12, 1\right)$ we can apply \cite[Theorem 1]{Simon} to $(\tveps)$ as a sequence in $L^2((0,T),H^{\beta}(\Omega))$ and obtain the strong convergence of $\tveps$ to $v_0$ in $L^2((0,T),H^{\beta}(\Omega))$.

Now we prove the convergence of $\te \veps$. It holds that
\begin{align*}
\Vert \te \veps - v_0 \Vert_{L^2((0,T)\times \Omega , \Ho) }^2 = \Vert \te \veps - v_0 \Vert_{L^2((0,T)\times \Omega,H^1(Y_1))}^2 + \Vert \te \veps - v_0 \Vert_{L^2((0,T)\times \Omega , H^1(\Gamma))}^2.
\end{align*}
We only prove the convergence to zero for the second term, since the first one can be treated in a similar way. We obtain from the properties of the unfolding operator from Lemma \ref{PropertiesUnfoldingOperator}, the trace inequality, and the inequality $\eqref{PropStarkeKonvergenzZshgGebiet}$
\begin{align*}
\Vert &\te \veps - v_0 \Vert_{L^2((0,T)\times \Omega,H^1(\Gamma))} \le C\Vert \te \veps - v_0 \Vert_{L^2((0,T)\times \Omega \times \Gamma)} + C\Vert \ngy \te \veps \Vert_{L^2((0,T)\times \Omega \times \Gamma)} 
\\
&\le C\left( \Vert \te \veps - v_0 \Vert_{L^2((0,T)\times \Omega \times Y_1)} + \epsilon \Vert  \nabla \veps \Vert_{L^2((0,T)\times \oeo)} + \epsilon^{\frac32} \Vert  \ngeps \veps \Vert_{L^2((0,T)\times \ge)}    \right)
\\
&\le C\left( \Vert \veps - v_0 \Vert_{L^2((0,T)\times \oeo)} + \Vert \te v_0 - v_0 \Vert_{L^2((0,T)\times \Omega \times Y_1)} + \epsilon    \right).
\end{align*}
The first term converges to zero for $\epsilon \to 0$, due to the strong convergence of $\tveps $ to $v_0$, and the second term because of \cite[Prop. 4.4]{CioranescuGrisoDamlamian2018}. This gives the desired result.
\end{proof}

\begin{remark}\mbox{}
\begin{enumerate}
[label = (\roman*)]
\item The extension of a bounded sequence $\veps$ in $L^2((0,T),H^1(\oeo))$ to $\tveps \in L^2((0,T),H^1(\Omega))$ preserving the boundedness is only possible for the connected domain $\oeo$. This  implies in a simple way the strong convergence of  $\tveps$, if there is also a control for the time-derivative. For the disconnected domain this method fails and therefore we use a Kolmogorov-Simon-compactness result for the unfolded sequence, where we  need an additional condition to control the shifts with respect to the macroscopic variable, see Theorem \ref{MainResultsGeneralStrongTSConvergence}.
\item If we replace in the assumptions of Proposition \ref{PropStarkeKonvergenzZshgGebiet} the space $\Hoeps$ by $H^1(\oeo)$,  we obtain the well known result that a bounded sequence in $L^2((0,T),H^1(\oeo))\cap H^1((0,T),H^1(\oeo)')$ has an extension converging strongly in $L^2((0,T)\times \Omega)$, see \cite{MeirmanovZimin}, and also in $L^2((0,T),H^{\beta}(\Omega))$ for $\beta \in \left(\frac12,1\right)$, see \cite{Gahn}. The proof above for the strong convergence of $\tveps$ can be easily adapted to this situation. However, we emphasize that  in our situation we cannot guarantee for the solution $(\uepso,\uepst)$ that $\partial_t \uepso \in L^2((0,T),H^1(\oeo)')$.
\end{enumerate}

\end{remark}

\subsection{The disconnected domain $\oet$}
\label{SectionDisconnectedDomain}

In this section we give a strong two-scale compactness result for the disconnected domain $\oet$ of Kolmogorov-Simon-type, \ie  it is based on \textit{a priori} estimates for the difference of discrete shifts, see condition \ref{StrongTSConvergenceConditionShifts} in Theorem \ref{MainResultsGeneralStrongTSConvergence}. As already mentioned above it is in general not possible to find an extension for a function in $H^1(\oet)$ to the whole domain $\Omega$ which preserves the \textit{a priori} estimates. Hence, the method from Section \ref{SectionConnectedDomain} for the connected domain fails. Therefore, we consider the unfolded sequence in the Bochner space $L^p(\Omega,L^2((0,T),\H_2^{\beta}))$ with $\beta \in \left(\frac12,1\right)$ and $p \in (1,2)$, and apply the Kolmogorov-Simon-compactness result from \cite{GahnNeussRaduKolmogorovCompactness}, which gives an extension of \cite[Theorem 1]{Simon} to higher-dimensional domains of definition. Here, a crucial point is the estimate for the shifts. An important reason to work here with general Bochner spaces, \ie Banach-valued functions spaces, is that we are dealing with manifolds and therefore linear shifts with respect to the space variable are not well defined.

 In the following Lemma we estimate the shifts of the unfolded sequence with respect to the macroscopic variable by the shifts of the function itself, see again Section \ref{SectionAprioriEstimates} for the notations.

\begin{lemma}\label{EstimateShiftsUnfoldingAndSequence}
Let $\veps \in L^2((0,T)\times \oej)$ for $j \in \{1,2\}$ and $w_{\epsilon} \in L^2((0,T)\times \ge)$. Then, for $0 < h \ll 1$ , $|z|< h$, and $\epsilon $ small enough it holds that
\begin{align*}
\big\|\te \veps (t,x + z , y) - \te \veps \big\|_{L^2(0,T) \times \Omega_{2h}\times Y_j)}^2 &\le \sum_{k \in \{0,1\}^n} \|\delta \veps \|_{L^2((0,T)\times \oehj)}^2,
\\
\big\|\te w_{\epsilon} (t,x + z , y) - \te w_{\epsilon} \big\|_{L^2(0,T) \times \Omega_{2h}\times \Gamma)}^2 &\le \epsilon \sum_{k \in \{0,1\}^n} \|\delta w_{\epsilon} \|_{L^2((0,T)\times \geh)}^2,
\end{align*}
with $l = l(\epsilon,z,k) = k + \left[\frac{z}{\epsilon}\right]$.
\end{lemma}
\begin{proof}
The proof for a thin layer can be found in \cite[p. 709]{NeussJaeger_EffectiveTransmission} and can be easily extended to our setting. See also \cite{Gahn} for more details.
\end{proof}

\begin{theorem}
\label{MainResultsGeneralStrongTSConvergence}
Let $\veps \in L^2((0,T),\Hteps)\cap H^1((0,T),\Hteps')$  with:
\begin{enumerate}
[label = (\roman*)]
\item\label{StrongTSConvergenceConditionAprioriEstimate} It holds the estimate
\begin{align*}
\|\veps \|_{L^2((0,T),\Hteps)} +  \|\partial_t \veps \|_{L^2((0,T),\Hteps')} \le C. 
\end{align*}
\item\label{StrongTSConvergenceConditionShifts} For $0 < h \ll 1 $ and $l\in \Z^n$ with $|l\epsilon|< h$ it holds that
\begin{align*}
\|\delta \veps \|_{L^2((0,T),L^2(\oeht))} 
\overset{\epsilon l \to 0}{\longrightarrow} 0.
\end{align*}
\end{enumerate}
Then, there exists $v_0 \in L^2((0,T),L^2(\Omega))$, such that for $\beta \in \left(\frac12,1\right)$ and $p\in [1,2)$ it holds up to a subsequence that
\begin{align*}
\te \veps \rightarrow v_0 \quad \mbox{ in } L^p(\Omega,L^2((0,T),\H_2^{\beta})).
\end{align*}
Especially, $\veps$ and $\veps|_{\ge}$ converge strongly in the two-scale sense  to $v_0$ (with respect to $L^p$).
\end{theorem}

\begin{proof}
Our aim is to apply \cite[Corollary 2.5]{GahnNeussRaduKolmogorovCompactness} to  $(\te \veps )$ as a sequence in  $ L^p(\Omega, L^2((0,T),\H_2^{\beta}))$ for $p \in [1,2)$ and $\beta \in \left(\frac12,1\right)$. Hence, we have to check the following three conditions:
\begin{enumerate}[label=(K\arabic*)]
\item\label{KolmogorovConditionRange} For every measurable set $A\subset \Omega$ the sequence $\left\{ \int_A \veps dx \right\}$ is relatively compact in $L^2((0,T),\H_2^{\beta})$,
\item\label{KolmogorovConditionShifts} for all $0<h \ll 1$ and $|z|<h$ it holds that 
\begin{align*}
 \sup_{\epsilon} \| \veps(\cdot + z) - \veps \|_{L^p(\Omega_h,B)} \rightarrow 0 \quad \mbox{ for } z \to 0,
\end{align*}
\item\label{KolmogorovConditionBoundary} for $h>0$ it holds that
$\, \sup_{ \epsilon} \int_{\Omega \setminus \Omega_h} |\veps(x)|^p dx \rightarrow 0 \, \mbox{ for } h \to 0$.
\end{enumerate}
We start with the condition \ref{KolmogorovConditionRange}. Let $A \subset \Omega$ be measurable and we define $V_{\epsilon}(t,y):= \int_A \te \veps (t,x,y)dx$. To show the relative compactness of $(V_{\epsilon})$, we use again \cite[Theorem 1]{Simon} as in the proof of Proposition \ref{StrongConvergenceConnectedDomain}. First of all, due to our assumptions on $\veps$ and the properties of the unfolding operator, for $t_1,t_2 \in (0,T)$ it holds that
\begin{align*}
\left\| \int_{t_1}^{t_2} V_{\epsilon} dt  \right\|_{\H_2} \le  \|\te \veps \|_{L^2((0,T)\times \Omega, \H_2)} \le C.
\end{align*}
Due to the compact embedding $\H_2 \hookrightarrow \H_2^{\beta}$ we obtain that $\int_{t_1}^{t_2} V_{\epsilon} dt$ is relatively compact in $\H_2^{\beta}$.
 Further, for $0<s \ll 1$ we obtain with the estimates for $\veps$ and the trace inequality $\eqref{TraceInequality}$
\begin{align*}
\big\| &V_{\epsilon}(t + s,y) - V_{\epsilon}\big\|_{L^2((0,T-s),\H_2)} \le \big\|\te \veps (t + s,x,y) - \te \veps  \big\|_{L^2((0,T-s)\times \Omega, \H_2)}
\\
&\le C \|\veps (t+s,x) - \veps\|_{L^2((0,T-s) \times \oet)} + C\epsilon \|\nabla \veps (t+s,x) - \nabla \veps \|_{L^2((0,T-s) \times \oet)}
\\
&\hspace{3em} + C  \epsilon^{\frac32}\Vert \ngeps \veps(t+s,x) - \ngeps \veps \Vert_{L^2((0,T-s)\times \ge)}
\\
&\le C \left(s^{\frac14} + \epsilon\right),
\end{align*}
where for the last inequality we used  Lemma \ref{LemmaNikolskiiEmbedding} to estimate the first term in the line before by using the same arguments as for the inequality $\eqref{AuxiliaryEstimateShiftTime}$ in the proof of Proposition \ref{StrongConvergenceConnectedDomain}. Hence, \cite[Theorem 1]{Simon} implies that $(V_{\epsilon})$ is relatively compact in $L^2((0,T),\H_2^{\beta})$, \ie condition \ref{KolmogorovConditionRange}.

For \ref{KolmogorovConditionShifts} we fix $0< h \ll 1$ and choose $|z|< h$. Lemma \ref{EstimateShiftsUnfoldingAndSequence} with $l= k + \left[\frac{z}{\epsilon}\right]$ (see the definition of the difference $\delta$ in $\eqref{DefinitionDifference}$),  the conditions \ref{StrongTSConvergenceConditionAprioriEstimate} and \ref{StrongTSConvergenceConditionShifts}, as well as the trace inequality $\eqref{TraceInequality}$ imply
\begin{align*}
\big\| &\te \veps (t, x+ z,y) - \te \veps \big\|_{L^2(\Omega_{2h} , L^2((0,T),\H_2))} 
\\
&\le C \sum_{k \in \{0,1\}^n} \left( \|\delta \veps \|_{L^2((0,T)\times \oeht)} + \epsilon \|\delta \nabla \veps \|_{L^2((0,T)\times \oeht)} + \epsilon^{\frac32} \Vert \delta \ngeps \veps \Vert_{L^2((0,T)\times \geh)} \right)
\\
&\le C \left( \sum_{j \in \{0,1\}^n} \|\delta \veps \|_{L^2((0,T)\times \oeht)} + \epsilon \right) \overset{\epsilon,z \to 0}{\longrightarrow} 0.
\end{align*}
We show that this implies the uniform convergence in \ref{KolmogorovConditionShifts} with respect to $\epsilon$, see also \cite[p.710-711]{NeussJaeger_EffectiveTransmission} or \cite[p.1476-1477]{friesecke2002theorem}. Let $0 < \rho $. Due to our previous results there exist $0 < \epsilon_0, \, \delta_0$, such that for all $\epsilon \le \epsilon_0 $ and $\vert z\vert \le \delta_0$ it holds that
\begin{align}\label{UniformEstimateShifts}
\big\| \te \veps (t, x+ z,y) - \te \veps \big\|_{L^2(\Omega_{2h} , L^2((0,T),\H_2))}  \le \rho.
\end{align}
Since $\epsilon^{-1}  \in \N$, there are only finitely many elements $\epsilon_i$ with $i = 1,\ldots, N$, such that $\epsilon_0 < \epsilon_i$. For every $\epsilon_i$ there exists a $0 < \delta_i$, such that $\eqref{UniformEstimateShifts}$ is valid for $\epsilon = \epsilon_i$ and all $\vert z\vert \le \delta_i$. 
Choosing $\delta:= \max_{i=0,\ldots,N} \{\delta_i\}$, inequality $\eqref{UniformEstimateShifts}$ holds uniformly with respect to $\epsilon$ for all $\vert z \vert\le \delta$. This implies \ref{KolmogorovConditionShifts}.
For the last condition \ref{KolmogorovConditionBoundary} we   use the H\"older-inequality to obtain for $p \in [1,2)$ and $0< h \ll 1$
\begin{align*}
\big\|\te \veps \big\|_{L^p(\Omega \setminus \Omega_h, L^2((0,T),\H_2))} \le |\Omega \setminus \Omega_h|^{\frac{2-p}{2p}}\big\|\te \veps \big\|_{L^2((0,T)\times \Omega \setminus \Omega_h , \H_2)} \le Ch^{\frac{2-p}{2p}} \overset{h\to 0}{\longrightarrow}0,
\end{align*}
where we used again estimate \ref{StrongTSConvergenceConditionAprioriEstimate}. 
Now, \cite[Corollary 2.5]{GahnNeussRaduKolmogorovCompactness} implies the 
 the strong convergence of $\te \veps $ up to a subsequence in $L^p(\Omega,L^2((0,T),\H_2^{\beta}))$ to a limit function $v_0$. Lemma \ref{PropAequivalenzTSKonvergenzUnfolding}
   implies   the strong two-scale convergence of $\veps$ to the same limit. The fact $v_0 \in L^2((0,T),L^2(\Omega))$ follows from standard two-scale compactness results, see \cite{Allaire_TwoScaleKonvergenz}, based on the estimate \ref{StrongTSConvergenceConditionAprioriEstimate}.
\end{proof}

\begin{remark}
Theorem \ref{MainResultsGeneralStrongTSConvergence} and its proof remain valid if we replace $\Hteps$ and $\H^{\beta}$ by $H^1(\oet)$ and $H^{\beta}(Y_2)$.
\end{remark}

\section{Derivation of the macroscopic model}
\label{SectionDerivationMacroscopicModel}

The aim of this section is the derivation of the macroscopic model $\eqref{MacroscopicModel}$ from Theorem \ref{MainResultMacroscopicModel} for $\epsilon \to 0$. Therefore we make use of compactness results from Section \ref{SectionTwoScaleConvergenceUnfolding} and the \textit{a priori} estimates from Section \ref{SectionExistence}. In the following Proposition we collect the convergence results for the microscopic solution $\ueps= (\uepso,\uepst)$:
\begin{proposition}\label{PropositionConvergenceResults}
Let $\ueps = (\uepso,\uepst)$ be the microscopic solution of the problem $\eqref{MicroscopicModel}$. There exist 
\begin{align*}
u_0^1 \in L^2((0,T),H^1(\Omega)),\quad
u_1^1 \in L^2((0,T), \Ho / \R), \quad 
u_0^2 \in L^2((0,T)\times \Omega),
\end{align*}
\begin{subequations}
such that up to a subsequence it holds for $p\in [1,2)$ 
\begin{align}
\uepso &\rightarrow u_0^1 &\mbox{ strongly in the two-scale sense},
\label{TSConvMicSolConnectedDomainSol}
\\
\nabla \uepso &\rightarrow \nabla u_0^1 + \nabla_y u_1^1 &\mbox{ in the two-scale sense},
\label{TSConvMicSolConnectedDomainGradient}
\\
\uepso|_{\ge} &\rightarrow u_0^1 &\mbox{ strongly in the two-scale sense on } \ge ,
\label{TSConvMicSolConnectedDomainSurfaceSol}
\\
\ngeps \uepso &\rightarrow P_{\Gamma} \nabla u_0^1 + \ngy u_1^1 &\mbox{ in the two-scale sense on } \ge,
\label{TSConvMicSolConnectDomainSurfaceGradient}
\\
\uepst &\rightarrow u_0^2 &\mbox{ strongly in the two-scale sense in } L^p,
\label{TSConvMicSolDisconnectedDomainSol}
\\
\nabla \uepst &\rightarrow 0 &\mbox{ in the two-scale sense},
\label{TSConvMicSolDisconnectedDomainGradient}
\\
\uepst|_{\ge} &\rightarrow u_0^2 &\mbox{ strongly in the two-scale sense in } L^p \mbox{ on } \ge,
\label{TSConvMicSolDisconnectedDomainSurfaceSol}
\\
\ngeps \uepst &\rightarrow 0 &\mbox{ in the two-scale sense on } \ge.
\label{TSConvMicSolDisconnectedDomainSurfaceGradient}
\end{align}
\end{subequations}

\end{proposition}
\begin{proof}
The convergence results $\eqref{TSConvMicSolConnectedDomainSol}$ - $\eqref{TSConvMicSolConnectDomainSurfaceGradient}$ follow immediately from Proposition \ref{StrongConvergenceConnectedDomain}, Lemma \ref{LemmaTSConvergence},  and the \textit{a priori} estimates in Proposition \ref{PropositionAprioriEstimates}.
\\
For $\eqref{TSConvMicSolDisconnectedDomainSol} $ - $\eqref{TSConvMicSolDisconnectedDomainSurfaceGradient}$ we first notice that due to Lemma \ref{LemmaTSConvergence}  there exists $u_0^2 \in L^2((0,T)\times \Omega)$ and $u_1^2 \in L^2((0,T)\times \Omega, \H_2/\R)$, such that up to a subsequence
\begin{align*}
\uepst &\rightarrow u_0^2 &\mbox{ in the two-scale sense}&,
\\
\nabla \uepst &\rightarrow \nabla_y u_1^2 &\mbox{ in the two-scale sense}&,
\\
\uepst|_{\ge} &\rightarrow u_0^2 &\mbox{ in the two-scale sense on }& \ge,
\\
\ngeps \uepst|_{\ge} &\rightarrow \nabla_{\Gamma} u_1^2 &\mbox{ in the two-scale sense on }& \ge.
\end{align*}

For the strong two-scale convergence of $\uepst$ and $\uepst|_{\ge}$ we make use of Theorem \ref{MainResultsGeneralStrongTSConvergence}, where we have to check the conditions \ref{StrongTSConvergenceConditionAprioriEstimate} and \ref{StrongTSConvergenceConditionShifts}. The first one is just the \textit{a priori} estimate from Proposition \ref{PropositionAprioriEstimates}. For \ref{StrongTSConvergenceConditionShifts} we use Proposition \ref{PropositionEstimateShifts} to obtain for fixed $0< h \ll 1$ and $l\in \Z^n$ with $\epsilon|l|<h$
\begin{align*}
\|\delta \uepst \|_{L^2((0,T),L^2(\oeht))} \le C \left(\|\delta \uepso \|_{L^2((0,T),L^2(\oeho))} + \big\|\delta \big(u_{\epsilon,i}^2,u_{\epsilon,i,\ge}^2\big)\big\|_{\Ltepsh} + \epsilon \right).
\end{align*}
For the first term on the right-hand side we have
\begin{align*}
\|\delta \uepso \|_{L^2((0,T),L^2(\oeho))} \le \big\|\te \uepso (x + l\epsilon,y)  - \te \uepso\big\|_{L^2((0,T)\times \Omega_h \times Y_1)}.
\end{align*}
The right-hand side converges to zero, due to the strong two-scale convergence of $\uepso$, \ie the strong convergence of $\te \uepso$ in $L^2((0,T)\times \Omega \times Y_1)$, and the standard Kolmogorov-compactness theorem. For the $L^2$-norm of $\delta u_{\epsilon,i}^2$ in the second term we argue in a similar way, where we can use the strong two-scale convergence in the Assumption \ref{AssumptionInitialConditions}. For the norm of $\delta u_{\epsilon,i,\ge}^2$ we can    use the Kolmogorov-Simon-compactness result from \cite[Corollary 2.5]{GahnNeussRaduKolmogorovCompactness}, applied to the strong convergent sequence $\big(\te u_{\epsilon,i,\ge}^2\big) $ in $L^2((0,T)\times \Omega,L^2(\Gamma))$.

To prove $\eqref{TSConvMicSolDisconnectedDomainGradient}$ and $\eqref{TSConvMicSolDisconnectedDomainSurfaceGradient}$ we have to show that $u_1^2$ is constant with respect to $y$. Therefore, we choose $\peps(t,x):= \epsilon \phi\left(t,x,\fxe\right)$ with $\phi \in C^{\infty}_0((0,T)\times \Omega \times \overline{Y_2})$ (periodically extended in the last variable) as a test-function in $\eqref{VariationalEquationMicroscopicProblem}$ for $j=2$ and integrate with respect to time to obtain
\begin{align}
\begin{aligned}
\label{HilfsgleichungGrenzwertNullGradientUepsTwo}
\int_0^T\langle &\partial_t \uepst , \peps\rangle_{\Hteps',\Hteps} dt  + \int_0^T \int_{\oet} D_{\epsilon}^2\nabla \uepst \cdot \left( \epsilon \nabla_x \phi\left(t,x,\fxe\right) + \nabla_y \phi\left(t,x,\fxe\right)  \right)dxdt 
\\
&+ \epsilon \int_0^T \int_{\ge}  D_{\ge}^2 \ngeps \uepst \cdot \left( \epsilon P_{\ge} \nabla_x \phi\left(t,x,\fxe\right) + P_{\ge} \nabla_y \phi\left(t,x,\fxe\right) \right) d\sigma dt
\\
=&  \epsilon \int_0^T \int_{\oet} f_{\epsilon}^2 (\uepst) \phi\left(t,x,\fxe\right) dx dt + \epsilon^2 \int_0^T \int_{\oet} h_{\epsilon}^2(\uepso,\uepst) \phi\left(t,x,\fxe\right)dx dt.
\end{aligned}
\end{align}
For the first term on the left-hand side including the time-derivative we get by integration by parts in time
\begin{align*}
\int_0^T \langle \partial_t \uepst , \peps \rangle_{\Hteps',\Hteps} dt = -\epsilon&\int_0^T \int_{\oet} \partial_t \phi\left(t,x,\fxe\right) \uepst dx dt 
\\
&- \epsilon^2 \int_0^T \int_{\ge} \partial_t \phi\left(t,x,\fxe\right)\uepst d\sigma dt.
\end{align*}
The right-hand side is of order $\epsilon$, due to the estimates in Proposition \ref{PropositionAprioriEstimates}. Hence, all terms in $\eqref{HilfsgleichungGrenzwertNullGradientUepsTwo}$ except the terms including the $\nabla_y$ are of order $\epsilon$ (again because of Proposition \ref{PropositionAprioriEstimates}) and we obtain for $\epsilon \to 0$
\begin{align*}
\int_0^T \int_{\Omega} \int_{Y_2} &D^2(y) \nabla_y u_1^2(t,x,y) \cdot \nabla_y \phi(t,x,y) dy dx dt 
\\
&+ \int_0^T \int_{\Omega}\int_{\Gamma} D_{\Gamma}^2(y) \ngy u_1^2(t,x,y) \cdot \ngy \phi(t,x,y) d\sigma_y dx dt = 0 .
\end{align*}
Due to the density of $C^{\infty}(\overline{Y_2})$ in $\H_2$, see \cite[Lemma 2.1]{Gahn2019}, the equation above holds for all $\phi \in L^2((0,T)\times \Omega,\H_2)$. This implies $u_1^2 = 0$. The Proposition is proved.

\end{proof}

We have the following representation of $u_1^1$:

\begin{corollary}\label{CorollaryRepresentation}
Almost everywhere in $(0,T)\times \Omega \times Y_1$ it holds that
\begin{align}
\label{RepresentationUoo}
u_1^1(t,x,y) = \sum_{i=1}^n \partial_{x_i} u_0^1(t,x) w_i^1(y),
\end{align}
where $w_i^1 \in \Ho/\R$ with $Y$-periodic boundary conditions is the unique weak solution of the following cell problem ($i=1,\ldots, n$)
\begin{align*}
-\nabla_y \cdot \big(D^1 (\nabla_y w_i^1 + e_i) \big) &= 0 &\mbox{ in } Y_1,
\\
-D^1(\nabla_y w_i^1 + e_i ) \cdot \nu &= - \ngy \cdot \big(D_{\Gamma}^1 (\ngy w_i^1 + \ngy y_i )\big) &\mbox{ on } \Gamma,
\\
w_i^1 \mbox{ is }& Y\mbox{-periodic and } \int_{\Gamma} w_i^1 d\sigma = 0.
\end{align*}
\end{corollary}

\begin{proof}
The procedure is quite standard, see \eg \, \cite{Allaire_TwoScaleKonvergenz}, but for the sake of completeness we give the main steps. We choose $\peps(t,x) = \epsilon \phi\left(t,x,\fxe\right)$ with $\phi \in C^{\infty}_0 ((0,T)\times \Omega , C_{\per}^{\infty}(\overline{Y_1}))$ as  a test-function in $\eqref{VariationalEquationMicroscopicProblem}$ and integrate with respect to time to obtain $\eqref{HilfsgleichungGrenzwertNullGradientUepsTwo}$ if we replace $j=2$ by $j=1$. From Proposition \ref{PropositionConvergenceResults} we get for $\epsilon \to 0$
\begin{align*}
0 = &\int_0^T \int_{\Omega} \int_{Y_1} D^1(y) \left[\nabla_x u_0^1(t,x) + \nabla_y u_1^1(t,x,y) \right] \cdot \nabla_y \phi(t,x,y)dy dx dt
\\
&+ \int_0^T \int_{\Omega} \int_{\Gamma} D_{\Gamma}^1(y) \left[ P_{\Gamma} (y) \nabla_x u_0^1(t,x) + \ngy u_1^1(t,x,y) \right]\cdot \ngy \phi(t,x,y) d\sigma_y dx dt.
\end{align*}
Due to the Lax-Milgram Lemma this problem has a unique solution $u_1^1$ and due to its linearity we easily obtain the representation $\eqref{RepresentationUoo}$.
\end{proof}

%
%
%
%

Now, we are able to formulate the macroscopic model. We show that $u_0=(u_0^1,u_0^2)$ from Proposition \ref{PropositionConvergenceResults} is the unique weak solution (the definition of a weak solution is given below) of the macro-model
%
\begin{align}
\begin{aligned}\label{MacroscopicModel}
(|Y_1| + |\Gamma|) \partial_t u_0^1 - \nabla \cdot \big(\hd^1 \nabla u_0^1 \big) &= \int_{Y_1} f^1(t,y,u_0^1) dy + \int_{\Gamma} h^1(t,y,u_0^1,u_0^2) d\sigma_y &\mbox{ in }& (0,T)\times \Omega,
\\
(|Y_2| + |\Gamma|) \partial_t u_0^2 &= \int_{Y_2} f^2(t,y,u_0^2) dy + \int_{\Gamma} h^2(t,y,u_0^1,u_0^2) d\sigma_y &\mbox{ in }& (0,T)\times \Omega,
\\ 
-\hd^1\nabla u_0^1 \cdot \nu &= 0 &\mbox{ on }& (0,T)\times \partial \Omega,
\\
u_0^j(0) &= \frac{|Y_j| u_{0,i}^j + |\Gamma| u_{0,i,\Gamma}^j}{|Y_j| +  |\Gamma|} &\mbox{ in }& \Omega,
\end{aligned}
\end{align}
where the homogenized diffusion coefficient  $\hd^1\in \R^{n\times n}$ is defined by ($i,l= 1,\ldots ,n$)
\begin{align*}
\big(\hd^1\big)_{il}:= \int_{Y_1} D^1 (&\nabla_y w_i^1 + e_i ) \cdot (\nabla_y w_l^1 + e_l ) dy 
\\
&+ \int_{\Gamma} D_{\Gamma}^1 (\ngy w_i^1 + \ngy y_i )\cdot (\ngy w_l^1 + \ngy y_l ) d\sigma,
\end{align*}
and  $w_i^1 \in \Ho /\R$ (see Section \ref{SectionFunctionsSpaces} for the definition of this space)  for $i=1,\ldots,n$ are the solutions of the cell problems
\begin{align}
\begin{aligned}\label{CellProblems}
-\nabla_y \cdot \big(D^1 (\nabla_y w_i^1 + e_i) \big) &= 0 &\mbox{ in } Y_1,
\\
-D^1(\nabla_y w_i^1 + e_i ) \cdot \nu &= - \ngy \cdot \big(D_{\Gamma}^1 (\ngy w_i^1 + \ngy y_i )\big) &\mbox{ on } \Gamma,
\\
w_i^1 \mbox{ is }& Y\mbox{-periodic and } \int_{\Gamma} w_i^1 d\sigma = 0.
\end{aligned}
\end{align}
We say that $u_0=(u_0^1,u_0^2)$ is a weak solution of the macroscopic model, if 
\begin{align*}
u_0^1 &\in L^2((0,T),H^1(\Omega))\cap H^1((0,T),H^1(\Omega)'),
\\
u_0^2 &\in L^2((0,T)\times \Omega)\cap H^1((0,T),L^2(\Omega)), 
\end{align*}
the equation for $\partial_t u_0^2$ in $\eqref{MacroscopicModel}$ is valid in $L^2((0,T) \times \Omega)$, and for all $\phi \in H^1(\Omega)$ it holds almost everywhere in $(0,T)$
\begin{align*}
(|Y_1| + |\Gamma|) \big \langle \partial_t u_0^1 , \phi &\big\rangle_{H^1(\Omega)',H^1(\Omega)} + \int_{\Omega} \hd^1 \nabla u_0^1 \cdot \nabla \phi dx 
\\
&= \int_{\Omega} \int_{Y_1} f^1(y,u_0^1) \phi dy dx  + \int_{\Omega}\int_{\Gamma} h^1(y,u_0^1,u_0^2) \phi d\sigma_ydx,
\end{align*}
together with the initial conditions from $\eqref{MacroscopicModel}$.
%
\begin{theorem}\label{MainResultMacroscopicModel}
The limit function $u_0= (u_0^1,u_0^2)$ from Proposition \ref{PropositionConvergenceResults} is the unique solution of the macroscopic problem $\eqref{MacroscopicModel}$.
\end{theorem}
%
%
%
%

\begin{proof}
 We illustrate the procedure for $j=1$ (the  case $j=2$ follows by similar arguments, where the diffusion terms vanishes in the limit). As a test-function in $\eqref{VariationalEquationMicroscopicProblem}$ for $j=1$ we choose $\phi \in C^\infty_0([0,T)\times \overline{\Omega}) $ and integrate with respect to time.  By integration by parts in time we obtain
\begin{align*}
-\int_0^T& \int_{\oeo} \uepso \partial_t \phi dx dt - \epsilon \int_0^T \int_{\ge} \uepso \partial_t \phi d\sigma dt 
\\
&+ \int_0^T\int_{\oeo} D_{\epsilon}^1 \nabla \uepso \cdot \nabla \phi dx dt   + \epsilon \int_0^T \int_{\ge} D_{\ge}^1 \ngeps \uepso \cdot \ngeps \phi d\sigma dt 
\\
=& \int_0^T \int_{\oeo} f_{\epsilon}^1(\uepso) \phi dx dt + \epsilon \int_0^T \int_{\ge} h_{\epsilon}^1(\uepso,\uepst) \phi d\sigma dt
\\
&+ \int_{\oeo} u_{\epsilon,i}^1 \phi dx + \epsilon \int_{\ge} u_{\epsilon,i,\ge}^1 \phi d\sigma
.
\end{align*}
Using the convergence results from Proposition \ref{PropositionConvergenceResults}, Corollary \ref{CorollaryRepresentation}, and Lemma \ref{LemmaStrongTSSurface}, as well as the Assumption \ref{AssumptionInitialConditions} on the initial conditions, we obtain for $\epsilon \to 0$
\begin{align*}
-\big(|Y_1|+|\Gamma|\big)&  \int_0^T \int_{\Omega} u_0^1 \partial_t \phi dx dt  + \int_0^T \int_{\Omega} \hd^1 \nabla u_0^1 \cdot \nabla \phi dx dt
\\
&= \int_0^T \int_{\Omega} \int_{Y_1} f^1(u_0^1) \phi dy dx dt + \int_0^T\int_{\Omega} \int_{\Gamma} h^1(u_0^1,u_0^2) \phi d\sigma_y dx dt
\\
&\hspace{2em} + \int_{\Omega}|Y_1|u_{0,i}^1 \phi(0) dx + \int_{\Omega} |\Gamma|u_{0,i,\Gamma}^1  \phi(0) dx
\end{align*}
Choosing $\phi$ with compact support in $(0,T)$ we get $\partial_t u_0^1 \in L^2((0,T),H^1(\Omega)')$ (see also Remark \ref{RemarkEnd})  with $u_0^1(0) = \frac{|Y_1| u_{0,i}^1 + |\Gamma| u_{0,i,\Gamma}^1}{|Y_1| +  |\Gamma|}$, and by density we obtain that $u_0^1$ is a weak solution of the macroscopic equation for $j=1$ in $\eqref{MacroscopicModel}$.  Uniqueness follows by standard energy estimates.
\end{proof}

\begin{remark}\label{RemarkEnd}\mbox{}
\begin{enumerate}
[label = (\roman*)]
\item The regularity of the time-derivative is also directly obtained from the \textit{a priori} estimates in Proposition \ref{PropositionAprioriEstimates}. In fact, define for $0<h \ll 1$ and $v:(0,T)\rightarrow X$ for a Banach space $X$ the difference quotient for $t \in (0,T-h)$
\begin{align*}
\partial_t^h v (t):= \frac{v(t+h) - v(t)}{h}.
\end{align*}
Then  for all $\phi \in C^{\infty}_0((0,T),C^{\infty}(\overline{\Omega}))$ it holds, due to Proposition \ref{PropositionConvergenceResults} and the \textit{a priori} estimates for the time-derivative in Proposition \ref{PropositionAprioriEstimates}:
\begin{align*}
\langle \partial_t^h u_0^1 , \phi &\rangle_{L^2((0,T-h),H^1(\Omega)'),L^2((0,T-h),H^1(\Omega))} = \int_0^{T-h} \int_{\Omega} \partial_t^h u_0^1 \phi dx dt
\\
&= \lim_{\epsilon \to 0 } \frac{1}{\vert Y_1 \vert + \vert \Gamma \vert } \left( \int_0^{T-h} \int_{\oeo} \partial_t^h \uepso \phi dx  + \epsilon \int_{\ge} \partial_t^h \uepso \phi d\sigma dt \right)
\\
&= \lim_{\epsilon \to 0} \frac{1}{\vert Y_1 \vert + \vert \Gamma \vert }\int_0^T \langle \partial_t^h \uepso , \phi \rangle_{\Hoeps',\Hoeps} dt.
\\
&\le \lim_{\epsilon \to 0} \frac{1}{\vert Y_1 \vert + \vert \Gamma \vert } \Vert \partial_t^h \uepso \Vert_{L^2((0,T),\Hoeps')} \Vert \phi \Vert_{L^2((0,T),\Hoeps)}
\\
&\le C \lim_{\epsilon \to 0} \Vert \partial_t \uepso \Vert_{L^2((0,T),\Hoeps')} \Vert \phi \Vert_{L^2((0,T),\Hoeps)}
\\
&\le C \lim_{\epsilon \to 0} \Vert \phi \Vert_{L^2((0,T),\Hoeps)} \le C \Vert \phi \Vert_{L^2((0,T),H^1(\Omega))},
\end{align*}
where at the end we used that $P_{\Gamma}$ is an orthogonal projection.
By density and the reflexivity of $L^2((0,T-h),H^1(\Omega))$ we obtain the boundedness
\begin{align*}
\Vert \partial_t^h u_0^1 \Vert_{L^2((0,T-h),H^1(\Omega)')} \le C,
\end{align*}
for a constant $C$ independent of $h$. This implies $\partial_t u_0^1 \in L^2((0,T),H^1(\Omega)')$. A similar argument implies $\partial_t u_0^2 \in L^2((0,T),H^1(\Omega)')$. However, the limit equation for $u_0^2$ even improves the regularity of $\partial_t u_0^2$.

\item We can also consider  the case of a connected-connected porous medium (for $n\geq 3$ and a domain $\Omega$ which can be decomposed in microscopic cells, for example a rectangle with integer side length, and an additional boundary condition on $\partial \ge$ is needed). In this case both macroscopic solutions are described by a reaction-diffusion equation as for $u_0^1$ in Theorem \ref{MainResultMacroscopicModel}. The derivation of the macroscopic model for the connected-connected case even gets simpler, because we only need the \textit{a priori} estimates from Proposition \ref{PropositionAprioriEstimates} and the convergence results for the connected domain in Section \ref{SectionConnectedDomain}. The estimates for the shifts in Proposition \ref{PropositionEstimateShifts} are no longer necessary.

\item The results can be easily extended to systems, see \cite{GahnDissertation} for more details.
\end{enumerate}
\end{remark}

\section{Discussion}

By the methods of two-scale convergence and the unfolding operator we derived a macroscopic model for a reaction-diffusion equation in a connected-disconnected porous medium with a nonlinear dynamic Wentzell-interface condition across the interface. The crucial point was to pass to the limit in the nonlinear terms, especially on the interface. Therefore, we established strong two-scale compactness results just depending on \textit{a priori} estimates for the sequence of solutions. For the proof we used the unfolding operator and a Banach-valued Kolmogorov-Simon-compactness argument, which was necessarily  for the disconnected domain. In fact, while the solutions in the connected domain $\oeo$ can be extended to the whole domain $\Omega$ preserving the \textit{a priori} estimates, this is not possible anymore for the disconnected domain.

We emphasize that the strong compactness result in Theorem \ref{MainResultsGeneralStrongTSConvergence} is not restricted to our specific problem, but on the \textit{a priori} estimates and the estimates for the shifts for the sequence. Therefore it can be easily applied to other problems. Especially, the results above  can be extended to systems in an obvious way.

The time-derivative in the Wentzell-boundary condition on the interface $\ge$ regularizes the problem and leads to a simple variational structure with respect to the Gelfand-triple $(\Hjeps, \Ljeps,\Hjeps')$, see $\eqref{VariationalEquationMicroscopicProblem}$. Hence, the problem seems to be more complex regarding stationary interface conditions (neglecting the time-derivative). On the other hand, neglecting the diffusion term on the surface leads to an ordinary differential equation on the surface. Hence, we loose spatial regularity on the surface and therefore we have to replace the space $\Hjeps$ by  the function space $\left\{(\ueps,\veps) \in H^1(\oej)\times L^2(\ge) \, : \, \ueps\vert_{\ge} = \veps \right\} $ with norm $\sqrt{\Vert \ueps\Vert_{H^1(\oej)}^2 + \epsilon \Vert \veps \Vert_{L^2(\ge)}^2}$. For this choice it could be expected that  the methods in the paper  can be adapted to the case without surface diffusion.
Nevertheless, both cases should be considered in more detail and are part of my ongoing work.

\subsection*{Acknowledgements}
The author was  supported  by   the Odysseus program of the Research Foundation - Flanders FWO (Project-Nr. G0G1316N) and the   project  SCIDATOS  (Scientific  Computing  for  Improved  Detection  and  Therapy  of  Sepsis), which was  funded  by  the  Klaus Tschira Foundation, Germany (Grant number 00.0277.2015).

\appendix

\section{Two-scale convergence and unfolding operator}
\label{SectionAppendix}

We repeat the definition of the two-scale convergence and the unfolding operator and summarize some well known properties and compactness results.

\subsection{Two-scale convergence}
\label{SectionTSConvergence}

In the following, unless  stated otherwise, we assume that $p \in (1,\infty)$ and $p'$ is the dual exponent of $p$. We start with the definition of the two-scale convergence, see \cite{Allaire_TwoScaleKonvergenz,Nguetseng}.

\begin{definition}
We say the sequence  $\ueps \in L^p((0,T)\times \Omega)$  converges in the two-scale sense (in $L^p$) to a limit function $u_0 \in L^p((0,T)\times Y)$, if for all $\phi \in L^{p'}((0,T)\times \Omega,C^0_{\per}(Y))$ it holds that
\begin{align*}
\lim_{\epsilon \to 0 } \int_0^T \int_{\Omega} \ueps(t,x) \phi\left(t,x,\fxe\right) dx dt = \int_0^T \int_{\Omega}\int_{Y} u_0(t,x,y) \phi(t,x,y) dx dy dt.
\end{align*}
We say the sequence converges strongly in the two-scale sense (in $L^p$), if it holds that
\begin{align*}
\lim_{\epsilon \to 0 } \|\ueps \|_{L^p((0,T)\times \Omega)} = \|u_0\|_{L^p((0,T)\times \Omega \times Y)}.
\end{align*}

\end{definition}

\begin{remark}\label{BemerkungTSKonvergenzNotation}\mbox{}
\begin{enumerate}
[label = (\roman*)]
\item For sequences in $L^p((0,T)\times \oej)$ on the perforated domain we also use the designation "two-scale convergence". The definition is also valid for such functions by extension by zero (or with the extension operator from \cite{Acerbi1992}), and considering suitable test-functions.
\item The two scale convergence introduced above should actually be referred to as "weak two scale convergence". However, in accordance with the definition in \cite{Allaire_TwoScaleKonvergenz} we neglect the word "weak" and only use "strong" to highlight the "strong two-scale convergence".
\item For the "two-scale convergence in $L^2$" we just write "two-scale convergence".
\end{enumerate}
\end{remark}
Next, we give the definition of the two-scale convergence on oscillating surfaces, see \cite{AllaireDamlamianHornung_TwoScaleBoundary,Neuss_TwoScaleBoundary}.
\begin{definition}
We say the sequence $\ueps \in L^p((0,T)\times \ge)$ converges in the two-scale sense (in $L^p$) to a limit function $u_0\in L^p((0,T)\times \Omega \times \Gamma)$, if for all $\phi \in C^0([0,T]\times \overline{\Omega},C_{\per}^0(\Gamma))$ it holds that
\begin{align*}
\lim_{\epsilon\to 0} \epsilon\int_0^T \int_{\ge} \ueps(t,x) \phi\left(t,x,\fxe\right) d\sigma_x dt = \int_0^T \int_{\Omega}\int_{\Gamma} u_0(t,x,y) \phi(t,x,y) d\sigma_y dx dt.
\end{align*}
We say the sequence converges strongly in the two-scale sense, if it holds that
\begin{align*}
\lim_{\epsilon \to 0 } \epsilon^{\frac{1}{p}}\|\ueps \|_{L^p((0,T)\times \ge)} = \|u_0\|_{L^p((0,T)\times \Omega \times \Gamma)}.
\end{align*}
In accordance with Remark \ref{BemerkungTSKonvergenzNotation}, we proceed analogously  for the two-scale convergence on $\ge$ and neglect the word "weak" and the addition "$L^2$".
\end{definition}

To pass to the limit $\epsilon \to 0$ in the diffusion terms in the bulk domain $\oej$ and the surface $\ge$ in the microscopic equation $\eqref{VariationalEquationMicroscopicProblem}$ we need compactness results for the spaces $\Hjeps$. In the following Lemma we summarize some weak two-scale compactness results for such functions, which can be found in \cite{Gahn2019}:
\begin{lemma}\label{LemmaTSConvergence}\mbox{}
For $j\in \{1,2\}$ let $\uepsj \in L^2((0,T),\Hjeps)$ be a sequence with
\begin{align*}
\|\uepsj\|_{L^2((0,T),\Hjeps)}  \le C.
\end{align*}
Then it holds:
\begin{enumerate}
[label = (\roman*)]
\item For $j=1$ there exist $u_0^1 \in L^2((0,T),H^1(\Omega))$ and a $Y$-periodic function $u_1^1 \in L^2((0,T)\times \Omega , \H_1/\R)$, such that up to a subsequence
\begin{align*}
\uepso &\rightarrow u_0^1 &\mbox{ in the two-scale sense}&,
\\
\nabla \uepso &\rightarrow \nabla_x u_0^1 + \nabla_y u_1^1 &\mbox{ in the two-scale sense}&,
\\
\uepso|_{\ge} &\rightarrow u_0^1 &\mbox{ in the two-scale sense on }&  \ge,
\\
\ngeps \uepso|_{\ge} &\rightarrow P_{\Gamma} \nabla u_0^1 + \ng u_1^1|_{\Gamma} &\mbox{ in the two-scale sense on }& \ge.
\end{align*}
\item For $j=2$ there exist $u_0^2 \in L^2((0,T)\times \Omega )$ and $u_1^2 \in L^2((0,T)\times \Omega,\H_2/\R)$ such that up to a subsequence
\begin{align*}
\uepst &\rightarrow u_0^2 &\mbox{ in the two-scale sense}&,
\\
\nabla \uepst &\rightarrow \nabla_y u_1^2 &\mbox{ in the two-scale sense}&,
\\
\uepst|_{\ge} &\rightarrow u_0^2 &\mbox{ in the two-scale sense on }& \ge,
\\
\ngeps \uepst|_{\ge} &\rightarrow \nabla_{\Gamma} u_1^2 &\mbox{ in the two-scale sense on }& \ge.
\end{align*}
\end{enumerate}
\end{lemma}

\subsection{The unfolding operator}

In the following we give the definition of the unfolding operator and summarize some well known properties, see the monograph \cite{CioranescuGrisoDamlamian2018}  for an overview about this topic, and also \cite{ArbogastDouglasHornung,BourgeatLuckhausMikelic,CioranescuDamlamianDonatoGrisoZakiUnfolding,Cioranescu_Unfolding2,VogtHomogenization}. In the following we consider the tuple $(G_{\epsilon},G) \in \{(\Omega,Y), (\oeo,Y_1),(\oet,Y_2),(\ge,\Gamma)\}$ and we define
\begin{align*}
\widehat{G}_{\epsilon}:= \mathrm{int} \bigcup_{k \in K_{\epsilon}} \epsilon \left( \overline{G} + k\right), \quad\quad \Lambda_{\epsilon}:= \Omega \setminus \overline{\widehat{G}_{\epsilon}}.
\end{align*}
Then, for $p\in (1,\infty)$ we define the unfolding operator  
\begin{align*}
\te : L^p((0,T)\times G_{\epsilon}) \rightarrow L^p((0,T)\times \Omega \times G),
\end{align*}
with
\begin{align*}
\te (\peps)(t,x,y):= \begin{cases}
\peps\left(t,\epsilon \left[\fxe\right] + \epsilon y \right) &\mbox{ for } x \in \widehat{G}_{\epsilon},
\\
0 &\mbox{ for } x \in \Lambda_{\epsilon}.
\end{cases}
\end{align*}
We emphasize that we use the same notation for the unfolding operator for the different choices of the tuple $(G_{\epsilon},G)$. It should be clear from the context in which sense it has to be understood. Further, we mention that unfolding operator commutes with the trace operator in the following sense: For $\peps \in L^p((0,T),W^{1,p}(\oej))$ for $j\in \{1,2\}$ it holds that
\begin{align*}
\te \big(\peps|_{\ge}\big) = \big(\te (\peps)\big)|_{\Gamma}.
\end{align*}

\begin{lemma}\label{PropertiesUnfoldingOperator}\mbox{}
\begin{enumerate}
[label = (\alph*)]
\item\label{PropertiesUnfoldingOperatorDomain} For $(G_{\epsilon},G)\in \{(\Omega,Y),(\oeo,Y_1),(\oet,Y_2)\}$ we have:
\begin{enumerate}
[label = (\roman*)]
\item \label{PropertiesUnfoldingOperatorDomainNorm}  For $\peps \in L^p((0,T)\times G_{\epsilon})$ it holds that
\begin{align*}
\| \te \peps \|_{L^p((0,T)\times \Omega \times G)} = \|\peps\|_{L^p((0,T)\times \widehat{G}_{\epsilon})}.
\end{align*}
\item For $\peps \in L^p((0,T),W^{1,p}(G_{\epsilon}))$ it holds that
\begin{align*}
\nabla_y \te \peps = \epsilon \te \nabla_x \peps.
\end{align*}
\end{enumerate}
\item\label{PropertiesUnfoldingOperatorSurface} For the unfolding operator on the surface we have:
\begin{enumerate}
[label = (\roman*)]
\item \label{PropertiesUnfoldingOperatorSurfaceNorm} For $\peps \in L^p((0,T)\times \ge)$ it holds that
\begin{align*}
\|\te \peps \|_{L^p((0,T)\times \Omega \times \Gamma)} = \epsilon^{\frac{1}{p}} \|\peps\|_{L^p((0,T)\times \ge)}.
\end{align*}
\item \label{PropertiesUnfoldingOperatorSurfaceGradient} For $\peps \in L^p((0,T), W^{1,p}(\ge))$ it holds that
\begin{align*}
\ngy \te \peps = \epsilon \te \ngeps \peps.
\end{align*}
\end{enumerate}
\end{enumerate}
\end{lemma}
\begin{proof}
For \ref{PropertiesUnfoldingOperatorDomain} and \ref{PropertiesUnfoldingOperatorSurface}\ref{PropertiesUnfoldingOperatorSurfaceNorm} see \cite{CioranescuGrisoDamlamian2018}. A proof for \ref{PropertiesUnfoldingOperatorSurface}\ref{PropertiesUnfoldingOperatorSurfaceGradient} can be found in \cite{GrafPeterDiffusionOnSurfaces}.
\end{proof}
In the following Lemma we give an equivalent relation between the unfolding operator and the two-scale convergence. For a proof see for example \cite{BourgeatLuckhausMikelic,CioranescuDamlamianDonatoGrisoZakiUnfolding,CioranescuGrisoDamlamian2018}.
\begin{lemma}\label{PropAequivalenzTSKonvergenzUnfolding}
Let $p \in (1,\infty)$.
\begin{enumerate}
[label = (\alph*)]
\item For $(G_{\epsilon},G)\in \{(\Omega,Y),(\oeo,Y_1),(\oet,Y_2)\}$ and a sequence $\ueps \in L^p((0,T)\times G_{\epsilon})$, the following statements are equivalent:
\begin{enumerate}
\item $\ueps \rightarrow u_0$ weakly/strongly in the two-scale sense in $L^p$,
\item $\te \ueps \rightarrow u_0$ weakly/strongly in $L^p((0,T)\times \Omega \times G)$.
\end{enumerate}
\item  For a sequence $\ueps \in L^p((0,T)\times \ge)$ with $\epsilon^{\frac{1}{p}}\|\ueps\|_{L^p((0,T)\times \ge)} \le C$, the following statements are equivalent:
\begin{enumerate}
\item $\ueps \rightarrow u_0$ weakly/strongly in the two-scale sense on $\ge$ in $L^p$,
\item $\te \ueps \rightarrow u_0$ weakly/strongly in $L^p((0,T)\times \Omega \times \Gamma)$.
\end{enumerate}
\end{enumerate}
\end{lemma}

\bibliographystyle{abbrv}
\bibliography{literature}

\begin{thebibliography}{10}

\bibitem{Acerbi1992}
E.~Acerbi, V.~Chiad\`{o}, G.~D. Maso, and D.~Percivale.
\newblock An extension theorem from connected sets, and homogenization in
  general periodic domains.
\newblock {\em Nonlinear Analysis, Theory, Methods \& Applications},
  18(5):481--496, 1992.

\bibitem{Allaire_TwoScaleKonvergenz}
G.~Allaire.
\newblock Homogenization and two-scale convergence.
\newblock {\em SIAM J. Math. Anal.}, 23:1482--1518, 1992.

\bibitem{AllaireDamlamianHornung_TwoScaleBoundary}
G.~Allaire, A.~Damlamian, and U.~Hornung.
\newblock Two-scale convergence on periodic surfaces and applications.
\newblock {\em in Proceedings of the International Conference on Mathematical
  Modelling of Flow Through Porous Media, A. Bourgeat et al., eds., World
  Scientific, Singapore}, pages 15--25, 1996.

\bibitem{AllaireHutridurga2012}
G.~Allaire and H.~Hutridurga.
\newblock Homogenization of reactive flows in porous media and competition
  between bulk and surface diffusion.
\newblock {\em IMA Journal of Applied Mathematics}, 77:788--215, 2012.

\bibitem{AmarGianni2018}
M.~Amar and R.~Gianni.
\newblock Laplace-{B}eltrami operator for the heat conduction in polymer
  coating of electronic devices.
\newblock {\em Discrete Cont. Dyn. Sys. B}, 23(4):1739--1756, 2018.

\bibitem{Anguiano2020}
M.~Anguiano.
\newblock Existence, uniqueness and homogenization of nonlinear parabolic
  problems with dynamical boundary conditions in perforated media.
\newblock {\em Mediterr. J. Math.}, 2020.

\bibitem{ArbogastDouglasHornung}
T.~Arbogast, J.~Douglas, and U.~Hornung.
\newblock Derivation of the double porosity model of single phase flow via
  homogenization theory.
\newblock {\em SIAM J. Math. Anal.}, 27:823--836, 1990.

\bibitem{BourgeatLuckhausMikelic}
A.~Bourgeat, S.~Luckhaus, and A.~Mikeli\'c.
\newblock Convergence of the homogenization process for a double-porosity model
  of immiscible two-phase flow.
\newblock {\em SIAM J. Math. Anal.}, 27:1520--1543, 1996.

\bibitem{CioranescuDamlamianDonatoGrisoZakiUnfolding}
D.~Cioranescu, A.~Damlamian, P.~Donato, G.~Griso, and R.~Zaki.
\newblock The periodic unfolding method in domains with holes.
\newblock {\em SIAM J. Math. Anal.}, 44(2):718--760, 2012.

\bibitem{Cioranescu_Unfolding2}
D.~Cioranescu, A.~Damlamian, and G.~Griso.
\newblock The periodic unfolding method in homogenization.
\newblock {\em SIAM J. Math. Anal.}, 40:1585--1620, 2008.

\bibitem{CioranescuGrisoDamlamian2018}
D.~Cioranescu, G.~Griso, and A.~Damlamian.
\newblock {\em The periodic unfolding method}.
\newblock Springer, 2018.

\bibitem{CioranescuSJPaulin}
D.~Cioranescu and J.~S.~J. Paulin.
\newblock Homogenization in open sets with holes.
\newblock {\em J. Math. Pures et Appl.}, 71:590--607, 1979.

\bibitem{DonatoNguyen_HomogenizationOfDiffusion}
P.~Donato and K.~H.~L. Nguyen.
\newblock Homogenization of diffusion problems with a nonlinear interfacial
  resistance.
\newblock {\em NoDEA Nonlinear Differential Equations and Appl.},
  22(5):1345--1380, 2015.

\bibitem{friesecke2002theorem}
G.~Friesecke, R.~D. James, and S.~M{\"u}ller.
\newblock A theorem on geometric rigidity and the derivation of nonlinear plate
  theory from three-dimensional elasticity.
\newblock {\em Communications on Pure and Applied Mathematics: A Journal Issued
  by the Courant Institute of Mathematical Sciences}, 55(11):1461--1506, 2002.

\bibitem{GahnDissertation}
M.~Gahn.
\newblock {\em Derivation of effective models for reaction-diffusion processes
  in multi-component media}.
\newblock PhD thesis (Friedrich-Alexander-Universit\"at Erlangen-N\"urnberg);
  Shaker Verlag, 2017.

\bibitem{Gahn2019}
M.~Gahn.
\newblock Multi-scale modeling of processes in porous media - coupling
  reaction-diffusion processes in the solid and the fluid phase and on the
  separating interfaces.
\newblock {\em Discrete \& Continuous Dynamical Systems Series B},
  24(12):6511--6531, 2019.

\bibitem{GahnNeussRaduKolmogorovCompactness}
M.~Gahn and M.~Neuss-Radu.
\newblock A characterization of relatively compact sets in {$L^p(\Omega,B)$}.
\newblock {\em Stud. Univ. Babe\c{s}-Bolyai Math.}, 61(3):279--290, 2016.

\bibitem{GahnNeussRaduKnabnerEffectiveModelSubstrateChanneling}
M.~Gahn, M.~Neuss-Radu, and P.~Knabner.
\newblock Derivation of an effective model for metabolic processes in living
  cells including substrate channeling.
\newblock {\em Vietnam J. Math.}, 45(1--2):265--293, 2016.

\bibitem{Gahn}
M.~Gahn, M.~Neuss-Radu, and P.~Knabner.
\newblock Homogenization of reaction--diffusion processes in a two-component
  porous medium with nonlinear flux conditions at the interface.
\newblock {\em SIAM J. Appl. Math.}, 76(5):1819--1843, 2016.

\bibitem{Galdi}
G.~P. Galdi.
\newblock {\em An Introduction to the Mathematical Theory of the Navier--Stokes
  Equations}.
\newblock Springer Monogr. in Math., Springer-Verlag, New York, 2011.

\bibitem{GrafPeterHomogenizationManifolds}
I.~Graf and M.~A. Peter.
\newblock A convergence result for the periodic unfolding method related to
  fast diffusion on manifolds.
\newblock {\em C. R. Acad. Sci. Paris, Ser. I}, 352(6):485--490, 2014.

\bibitem{GrafPeterDiffusionOnSurfaces}
I.~Graf and M.~A. Peter.
\newblock Diffusion on surfaces and the boundary periodic unfolding operator
  with an application to carcinogenesis in human cells.
\newblock {\em SIAM J. Math. Anal.}, 46(4):3025--3049, 2014.

\bibitem{HewittStromberg}
E.~Hewitt and K.~Stromberg.
\newblock {\em Real and Abstract Analysis}.
\newblock Springer-Verlag, New York, 1975.

\bibitem{HornungJaegerMikelic}
U.~Hornung, W.~J\"ager, and A.~Mikeli\'c.
\newblock Reactive transport through an array of cells with semi-permeable
  membranes.
\newblock {\em ESAIM Mathematical Model. and Numer. Anal.}, 28:59--94, 1994.

\bibitem{MeirmanovZimin}
A.~Meirmanov and R.~Zimin.
\newblock Compactness result for periodic structures and its application to the
  homogenization of a diffusion-convection equation.
\newblock {\em Elec. J. of Diff. Equat.}, 115:1--11, 2011.

\bibitem{Neuss_TwoScaleBoundary}
M.~Neuss-Radu.
\newblock Some extensions of two-scale convergence.
\newblock {\em C. R. Acad. Sci. Paris S\'er. I Math.}, 322:899--904, 1996.

\bibitem{NeussJaeger_EffectiveTransmission}
M.~Neuss-Radu and W.~J\"ager.
\newblock Effective transmission conditions for reaction-diffusion processes in
  domains separated by an interface.
\newblock {\em SIAM J. Math. Anal.}, 39:687--720, 2007.

\bibitem{Nguetseng}
G.~Nguetseng.
\newblock A general convergence result for a functional related to the theory
  of homogenization.
\newblock {\em SIAM J. Math. Anal.}, 20:608--623, 1989.

\bibitem{Ptashnyk}
M.~Ptashnyk and T.~Roose.
\newblock Derivation of a macroscopic model for transport of strongly sorbed
  solutes in the soil using homogenization theory.
\newblock {\em SIAM J. Appl. Math.}, 70:2097--2118, 2010.

\bibitem{Simon}
J.~Simon.
\newblock Compact sets in the space ${L}^p(0,{T};{B})$.
\newblock {\em Ann. Mat. Pura Appl.}, 146:65--96, 1987.

\bibitem{VogtHomogenization}
C.~Vogt.
\newblock A homogenization theorem leading to a volterra integro-differential
  equation for permeation chromatography.
\newblock SFB 123, University of Heidelberg, Preprint 155 and Diploma-thesis,
  1982.

\bibitem{winkel2004metabolic}
B.~S. Winkel.
\newblock Metabolic channeling in plants.
\newblock {\em Annu. Rev. Plant Biol.}, 55:85--107, 2004.

\end{thebibliography}

\end{document}